\documentclass[11pt]{article}

\usepackage{amssymb}
\usepackage{amsthm, bbm}
\usepackage{amsfonts}
\usepackage{amsmath}
\usepackage{pmboxdraw}
\usepackage{verbatim} 
\usepackage{graphicx}
\usepackage{color, mathtools, textcomp}
\usepackage[colorlinks=true, citecolor=blue, filecolor=cyan, linkcolor=red, urlcolor=magenta]{hyperref}

\usepackage{cite}

\usepackage{amssymb,amsmath,latexsym, bbm}
\usepackage{epsfig, color}

\newtheorem{thm}{Theorem}[section]
\newtheorem{defn}[thm]{Definition}
\newtheorem{cor}[thm]{Corollary}

\newtheorem{lem}[thm]{Lemma}
\newtheorem{propn}[thm]{Proposition}
\newtheorem{remark}[thm]{Remark}
\newtheorem{example}[thm]{Example}

\newcommand{\pf}{\noindent{\bf Proof.} }
\def\qed{{\hfill $\Box$ \bigskip}}

\def\sD {{\cal D}} \def\sE {{\cal E}} \def\sF {{\cal F}}
 \def\sH {{\cal H}} 
  \def\sL {{\cal L}}
 \def\sN {{\cal N}}

\def\P{\mathbb{P}}
\def\E{\mathbb{E}}

\def\EE{\mathcal{E}}
\def\FF{\mathcal{F}}
\def\NN{\mathcal{N}}

\def\R{{\mathbb R}}
\def\bar{\overline}
\def\wh{\widehat}
\def\wt{\widetilde}
\def\eps{\varepsilon}
\def\bR {{\mathbb R}}
\def\bE {{\mathbb E}}
\def\bea{\begin{align*}}
\def\eea{\end{align*}}
\def\bee{\begin{equation}}
\def\eee{\end{equation}}
\def\bP {{\mathbb P}}

\def\nn{{\nonumber}}

\def\1{{\mathbbm 1}}

\def\<{\langle}
\def\>{\rangle}
\def\wt{\widetilde}

\makeatletter
\@addtoreset{equation}{section}

\makeatother

\begin{document}
\bibliographystyle{plain}

\title{\Large \bf
 Heat kernel upper bounds for symmetric Markov semigroups}
\author{Zhen-Qing Chen \quad Panki Kim \quad Takashi Kumagai
\quad Jian Wang}

\maketitle

\begin{abstract}
 It is well known that Nash-type inequalities for symmetric Dirichlet forms are equivalent
 to on-diagonal heat kernel upper bounds
 for the
 associated symmetric Markov semigroups. In this paper, we
 show that
 both imply (and hence are equivalent to)
 off-diagonal heat kernel upper bounds under some mild assumptions.
 Our approach is based on a new generalized
 Davies' method.
 Our  results
 extend  that
 of
  \cite{CKS} for Nash-type inequalities with power order considerably and also extend  that
  of
  \cite{Gri}
 for second order differential operators on a complete non-compact manifold.



\end{abstract}

\section{Introduction}
Analysis on heat
kernel has been an active research area for a long time
in analysis, geometry and probability.
A  central topic of the area is
to obtain
global quantitative  bounds on the heat kernel.
 Nash-type inequalities can
 effectively
 characterize
 upper bounds of the operator norms
 of semigroups from $L^1$-space to $L^\infty$-space (i.e., on-diagonal heat kernel upper bounds).
For example, let $(P_t)_{ t\geq 0}$ be a symmetric Markov semigroup on $L^2(E;m)$ associated with a symmetric regular Dirichlet form $(\sE, \sF)$ on
$L^2(E; m)$, where $E$ is a locally compact separable metric space and $m$ is a $\sigma$-finite Radon measure on $E$ with full support.
It is well known that (see \cite[Theorem 2.1]{CKS})
\begin{equation}\label{e:1-01} \|P_t\|_{1\to\infty}\le \varphi(t)e^{\delta t},\quad t>0
\end{equation}
 with
 $\varphi(t)=c \, t^{-\nu/2}$ for some constants $c,\nu>0$ and $\delta\ge0$,  is  equivalent to the following Nash-type
inequality
\begin{equation}\label{e:1-022}
\|f\|_2^{2+4/\nu}\le C(\sE(f,f)+\delta\|f\|_2^2)\|f\|_1^{4/\nu}   \quad \hbox{for } f\in \sF.
\end{equation}
Here and in what follows, $\|\cdot\|_p$ denotes the $L^p$-norm of the space $L^p(E;m)$ for  $p\in [1,\infty]$, and
$\| T\|_{p\to q}$ denotes the operator norm of a bounded operator $T: L^p(E; m)\to L^q(E; m)$ for $p, q\in [1, \infty]$.
  For $\varphi(t)=c\, t^{-\nu/2}$, we refer
 the reader
  to
  \cite{N} for the original proof that \eqref{e:1-022} implies \eqref{e:1-01} when $\delta=0$,  to \cite{V}
  for the connection
    between the estimate \eqref{e:1-01} for $\nu >2$ and  Sobolev inequalities,
  and to \cite{BCLS} for the equivalence between
  Sobolev-type inequalities and Nash-type inequalities.
  When  $\delta=0$ and $\varphi(t)$ is a strictly  decreasing
  differential bijection on
  $\R_+:=(0,\infty)$
  so that the derivative of $-\log \varphi(t)$
  is of polynomial growth,
  Coulhon   \cite[Theorem II.5]{Cou}   showed that
  \eqref{e:1-01} is equivalent to
the following generalized Nash-type inequality
\begin{equation}\label{e:1-02}
\theta(\|f\|_2^2)\le  \sE(f,f)   \quad \hbox{for } f\in \sF  \hbox{ with } \|f\|_1\le 1,
\end{equation}
where $\theta(r)$ is a positive function on
 $\R_+$ determined by $\varphi(r)$; see  Theorem \ref{T:thm1} or Section \ref{Section3}
for details.

The main goal of this paper is to establish the equivalences among Nash-type inequalities \eqref{e:1-02}, on-diagonal heat kernel upper bounds \eqref{e:1-01} and off-diagonal heat kernel upper bounds
for a large class of decreasing functions $\varphi (t)$ and for
general symmetric Markov processes which may have discontinuous sample paths.
 To state our main results, we need to introduce a class of regular functions which
  will serve
 as the rate function $\varphi(t)$ in \eqref{e:1-01}.

\begin{defn}\label{D1.1}\rm
A strictly deceasing $C^1$-function
$\varphi:
\R_+
\to \R_+$ is said to be \emph{regular}, if the following hold:
\begin{itemize}
\item[(i)] $\varphi(0):=\lim_{r\to0}\varphi(r)=\infty$ and $\varphi(\infty):=\lim_{r\to\infty}\varphi(r)=0$;
\item[(ii)]  there
exist a constant $c\ge1 $ and a decreasing function $N(t):\R_+\to\R_+$ such that,
with
 $M(t) :=-\log \varphi(t)$,
 $c^{-1} N(t)\le M'(t)\le cN(t)$
 on $\R_+$  and there is a constant $c_0>0$ so that $M'(t)\le c_0 M'(at)$ for all $t>0$ and $a\in [1,2]$;

\item[(iii)] for any $\lambda>2$ large enough, there exist positive constants $c(\lambda)$ and $C(\lambda)$ such that
$$\prod_{k=1}^\infty \left(\varphi \big((\lambda-1)\lambda^{-(k+1)}t\big)\right)^{2^{-k}}\le C(\lambda)\varphi\big(c(\lambda)t\big),\quad t>0.$$
\end{itemize}We denote this class of functions by $\mathcal{R}$. \end{defn}

Typical examples for regular functions  on $\R_+$ are $\varphi(r)=r^{-\nu/2}$ with $\nu>0$, or
  $\varphi(r)=r^{-\nu/2}$ with $\nu>0$ for small $r>0$, and $\varphi(r)=e^{-r^\alpha}$ with $\alpha\in (0,1]$ for large $r>0$.
  Indeed, for any decreasing function $\varphi:\R_+\to \R_+$ such that $\varphi(0)=\infty$, $\varphi(\infty)=0$ and that $1/\varphi$ has the doubling property (i.e., there is a constant $c_1\ge 1$ such that $1/ \varphi(2r)\le c_1/\varphi(r)$  for all $r>0$),
  we can find some $\bar\varphi\in \mathcal{R}$ and a constant $c_2\ge1$ so that
  $c_2^{-1}\bar\varphi(r)\le \varphi(r)\le c_2\bar\varphi(r)$ for all $r\in \R_+$; see Proposition \ref{P:dou}.

The following is a  special case of the main result,  Theorem \ref{T:3.2}, of this paper.

\begin{thm}\label{T:thm1} Suppose that $(\sE, \sF)$ is  a symmetric regular Dirichlet form on $L^2(E;m)$
having no killing
inside $E$
and $(P_t)_{ t\geq 0}$ is its associated semigroup.
Let $\varphi \in \mathcal{R}$
and $\delta\ge0$.
The following statements are equivalent.
\begin{itemize}
\item[{\rm(i)}] The following Nash-type inequality holds{\rm:}
\begin{equation}\label{E:Nash}
  c_1\theta(c_2 \|f\|_2^2)\le   \sE( f, f)+\delta\|f\|_2^2
 \qquad \hbox{for }f\in \sF \hbox{ with } \| f\|_1
\leq 1,
\end{equation}
 where $\theta(r)= -\varphi'(\varphi^{-1}(r))$ and $c_1, c_2$ are positive constants.

\item[{\rm(ii)}] The semigroup $(P_t)_{t\ge0}$ is ultracontractive in the sense that
 there are constants $c_3, c_4>0$ such that for all $t>0$,
\begin{equation}\label{e:ondup0}
\|P_t\|_{1\to\infty}\le c_3 \varphi(c_4 t)e^{\delta t}.
\end{equation}
\item[{\rm(iii)}] There are
a properly
exceptional set ${\cal N}\subset E$
and
 a  heat kernel
$p(t,x,y)$ associated with the semigroup $(P_t)_{t\ge0}$ and  defined on $(0,\infty)\times (E\setminus {\cal N})\times (E\setminus {\cal N})$ such that  for any $\varepsilon\in (0,1)$ there are
constants $C_\varepsilon, c_\epsilon>0$
 so that for all $t>0$
 and $x,y\in E\setminus {\cal N}$,
\begin{equation}\label{e:1.5}
p(t,x,y)\leq C_\varepsilon\, \varphi(
c_\varepsilon t) e^{
\delta t}  \;\exp\left( -|\psi(y)-\psi(x)|+
(1+\eps)\Lambda(\psi)^2     t \right),
\end{equation} where $$\Lambda (\psi)^2 :=
  \max\left\{\left\|\frac{d e^{-2\psi}\Gamma(e^{\psi}, e^\psi)}{d m}\right\|_\infty,
  \  \left\|\frac{d e^{2\psi}\Gamma(e^{-\psi}, e^{-\psi})}{d m}\right\|_\infty\right\} <\infty;
$$
equivalently, \begin{equation}\label{e:1.5--}
 p(t, x, y)\leq     C_\varepsilon
   \varphi(  c_\varepsilon t)  e^{ \delta t}
  \;  \exp \left( -\frac{d_{\EE} (x, y)^2}{
 4(1+\eps) t}\right),
\end{equation}
where
$$ d_\EE (x, y):=\sup\left\{\psi (x)-\psi(y): \ \psi\in
\sF\cap C_b(E)
 \hbox{ with } \Lambda (\psi) \leq 1\right\}.$$
\end{itemize}
\end{thm}

\medskip

In the statement above,  $\Gamma(f,f)$  is
the energy measure of $f$ for $\EE$ when  $f\in \sF$.
Precise definitions of energy measure, properly
exceptional set,
Dirichlet form having no killings inside $E$,
 and heat kernel can be found in Section \ref{section2.1}.

As mentioned above, the equivalence of (i) and (ii) in Theorem \ref{T:thm1}
 is known; see \cite[Theorem II.5]{Cou} for the case that $\delta=0$.
Clearly (iii) implies (ii).
So the new contribution of Theorem \ref{T:thm1} is that (i) implies (iii).
 When $\varphi(t)=c \, t^{-\nu/2} $ for some constants $c,\nu>0$ and $\delta\ge0$,
  that (i)  gives (iii)  is
  \cite[Theorem 3.25]{CKS}, which is the main result  of  \cite{CKS}.
 (From it, one can easily extend  the result to the case when
 $\delta=0$ and
 $\varphi(t)=c \, \left(  t^{-\nu/2} \1_{\{t\leq 1\}} + t^{-\mu /2} \1_{\{ t>1\} } \right)$ for some constants $c>0$
 and $0<\mu \leq \nu<\infty$;
 see \cite[Corollary 3.28]{CKS}.)
The approach of \cite{CKS}
 makes  use  of the condition
 $\varphi(t)=c \, t^{-\nu/2} $ (or equivalently,   \eqref{e:1-022}) in an essential way  (see, e.g.,  \cite[Lemma (3.21)]{CKS}), which does not seem to work for general    $\varphi\in \mathcal{R}$.
  We note that there are quite many  cases
  that the correct  on-diagonal heat kernel upper bounds are not of the form
  $c \, t^{-\nu/2}e^{ \delta t}$.
  For example, on-diagonal estimates for the heat kernel on a non-compact manifold with bounded geometry can be of the form
  $c_1 t/(\log t)^{1/\alpha}$
  for large time, when the volume of the ball grows at least with the rate $c_2e^{c_3 r^\alpha}$ with $\alpha\in (0,1]$ for large
  radius $r$
  (see \cite{BCG} or Example \ref{Exm1.2} below for more details); on-diagonal heat kernel upper bounds for mixtures of symmetric stable-like processes on $\R^d$   are
   of the form $c_4(\Phi^{-1}(t))^d$ for some
  strictly increasing weighted function
    $\Phi$
   which satisfies doubling and reverse doubling properties (see \cite{CK03,CK08,CKW1} or Example \ref{Exm1.4} below).

\medskip

There have been quite many efforts
in obtaining on off-diagonal heat kernel estimates
under
more general settings with more
general on-diagonal bounds.
Several authors have studied  Gaussian off-diagonal bounds for symmetric diffusions by using functional inequalities.
 Theorem \ref{T:thm1}  is related in spirit   to a beautiful result  of  Grigor'yan \cite[Theorem 1.1]{Gri} that, under suitable conditions,
an
isoperimetric-type  inequality (that is,  a
Faber-Krahn
type inequality),  an on-diagonal upper bound,
and an off-diagonal Gaussian upper estimate are equivalent up to some  constant multiplies for Brownian motion
on a complete non-compact manifold.  (Note that
these isoperimetric
 inequalities imply Nash-type inequalities or Sobolev inequalities, but the converse is not true.)
The  approach of \cite{Gri} relies
crucially on some constructions that are specific to the setting of second order differential operators, and uses some technical tools like a mean-value type theorem, which do not seem to be easily available in  general setting as in our paper.

For  strongly local Dirichlet forms,
by the Leibniz rule, $\Lambda (\psi)^2=d\Gamma(\psi,\psi)/dm$, and so we have the following statement.

\begin{cor} \label{C:local} Let $(\sE, \sF)$ be a strongly local, symmetric regular Dirichlet form on $L^2(E;m)$
having no killings inside $E$,
and $(P_t)_{ t\geq 0}$ its associated semigroup.
 Let $\varphi \in \mathcal{R}$ and $\delta\ge0$. Then {\rm(i)} and {\rm(ii)} in Theorem $\ref{T:thm1}$ are equivalent
 to the following
\begin{itemize}
\item[{\rm(iii)'}] There are
a properly
exceptional set ${\cal N}\subset E$
and  a heat kernel
$p(t,x,y)$ associated with the semigroup $(P_t)_{t\ge0}$ and  defined on $(0,\infty)\times (E\setminus {\cal N})\times (E\setminus {\cal N})$ such that for any $\varepsilon\in (0,1)$, there are positive constants $c_\varepsilon$ and $ C_\varepsilon$
 so that for $t>0$ and $x,y\in E\backslash {\cal N}$,
\begin{equation}\label{e:1.6}
 p(t, x, y)\leq      C_\varepsilon \varphi( c_\varepsilon t)e^{ \delta t}
  \;  \exp \left( -\frac{d_{\EE}^* (x, y)^2}{4(1+\varepsilon) t}\right),
\end{equation}
 where $d^*_\EE (x, y)$
is
the intrinsic distance
induced by the Dirichlet form $(\EE, \FF)$, i.e.,
$$
d^*_\EE (x, y):=\sup\left\{\psi (x)-\psi(y): \ \psi\in
\sF\cap C_b(E)  \hbox{ with } d\Gamma (\psi, \psi)/dm \leq 1 \right\}.
$$
\end{itemize}

\end{cor}

We note that in general the constant $\eps$
can not be taken to be zero in the exponent of
\eqref{e:1.6}.
See \cite[Example 3.1]{Mol}
 (especially the last asymptotic estimate there)
 for a counterexample, and
\cite{CouS}
and the references therein for sharp pointwise Gaussian estimates.

\medskip

Our approach adopts
Davies' idea of
 obtaining  off-diagonal heat kernel upper bounds
 through establishing $L^1\to L^\infty$-bounds for some perturbed semigroups.
 The main novel contributions of this paper   are

\begin{itemize}
 \item We directly apply
the generalized Nash-type inequalities to the perturbed semigroups
  by a localization trick.
We do not use the approach in  \cite[Lemma (3.21)]{CKS}  that is based on differential  inequalities,  which seems
to be specific to $\varphi (t)$  of being $c \, t^{-\nu/2}$ considered there.

 \item We introduce a new iteration argument to obtain the $L^1\to L^\infty$-bounds for perturbed semigroups,
 which works for general regular decay rate functions $\varphi\in \mathcal{R}$.
     \end{itemize}

In addition, our approach of establishing that
(i) implies (iii)  of
 Theorem \ref{T:thm1}
 is much simpler than that of \cite{CKS} even  in the  special case of
 $\varphi(t)=c \, t^{-\nu/2}$ for some $c,\nu>0$. As mentioned in \cite{Gri}, there
exist  many important classes of manifolds
on which the heat kernel of Brownian motion
decreases faster than polynomially as $t\to\infty$. Since Corollary \ref{C:local} holds for all regular functions, it can be used to deal with  the case that the heat kernel has an exponential/subexponential decay for strongly local Dirichlet forms, including a
 large variety of manifolds mentioned in \cite{Gri}.

 Davies'
method goes back to the paper \cite{Dav1} by E.B. Davies, where
 explicit constants in the exponential term for Gaussian upper bounds associated with various second
order elliptic or hypoelliptic operators on $\R^d$
were established. The method of \cite{Dav1}
uses logarithmic Sobolev inequality:
$$
\int f^2\log \frac{f^2}{\|f\|_2^2}\,dm\le r \sE(f,f)+(\log \beta(r)) \|f\|_2^2,
\quad f\in \sF, r>0.
$$
See \cite{Dav2} for more details on this topic.

Compared to the method using logarithmic Sobolev inequalities, ours
have some advantages. For example, our approach based on Nash-type inequalities
 \eqref{E:Nash}
can produce
 off-diagonal heat kernel upper bounds
with the optimal constant $4+\eps$
for the $d_\sE (x, y)^2$ term in  the exponent in \eqref{e:1.5--}
for general Dirichlet forms, which seems difficult to obtain through logarithmic Sobolev inequalities.
(Cf., \cite[Corollary 5.3]{Mi},
 where  $\delta $ should be zero  
 and
  $\varphi (t)$ is of the form  
  $ct^{-\nu/2} \ell_0 (t)$ for some positive
 function $\ell_0 $ on $(0, \infty)$ 
 that is slowly varying
 at both $0$ and $\infty$.)
On the other hand, as pointed out in \cite{Cou}, there are a number of works concerning the relations between Nash-type inequalities and other functional inequalities (for example, isoperimetric inequalities),
so it is possible to restate our main result in terms of other functional inequalities instead of
 Nash-type  inequalities.
 Here are two examples. First, we
give
the following corollary to
Theorem \ref{T:thm1}, which extends the main result of \cite{Gri} --- \cite[Theorem 1.1]{Gri} to general symmetric regular Dirichlet forms.
Recall that for an open set $D\subset E$, the part Dirichlet form $(\sE, \sF_D)$ in $D$ of $ (\sE, \sF)$ is given by
$$
\sF_D=\{u\in \sF: u = 0 \ \sE \hbox{-quasi-everywhere on } D\}.
$$
It is known (cf. \cite{CF, FOT}) that $(\sE, \sF_D)$ is a regular Dirichlet form on $L^2(D; m)$ and it corresponds
to the part process $X^D$ of $X$ associated with $(\sE, \sF)$ killed upon leaving $D$.

\begin{cor}\label{Cor1.4} Suppose that $(\sE, \sF)$ is  a symmetric regular Dirichlet form on $L^2(E;m)$
having no killing
inside $E$ and that  $(P_t)_{ t\geq 0}$ is its associated semigroup.
 Let $\varphi \in \mathcal{R}$ and
$\Theta(t) =N(\varphi^{-1}(1/t))$, where $N(t):\R_+\to \R_+$ is a decreasing function given in Definition $\ref{D1.1}${\rm(ii)}.
Then, {\rm(i)}, {\rm(ii)} and {\rm(iii)} in Theorem $\ref{T:thm1}$ with $\delta=0$ are equivalent
 to the following isoperimetric-type
 inequalities{\rm:}
 \begin{itemize}
\item[{\rm(iv)}] There exist constants $c_5,c_6>0$ such that for any pre-compact open set
$D\subset E$,
\begin{equation}\label{iso}
\lambda_1(D)\ge c_5 \Theta(c_6 m(D)),
\end{equation} where $\lambda_1(D)$ is the first eigenvalue of
the Dirichlet form $(\sE, \sF_D)$, i.e.,
$$
\lambda_1(D)=\inf\{\sE(f,f): f\in \sF\cap C_c(D) \hbox{ with } \|f\|_2=1\}.
$$

\item[{\rm(v)}] There exist constants $c_7,c_8>0$ such that for any pre-compact open set $D\subset E$ and $k\ge1$,
$$\lambda_k(D)\ge c_7 \Theta(c_8 m(D)/k),
$$
where $\lambda_k(D)$ is the $k^{\rm th}$
eigenvalue of
the Dirichlet form $(\sE,\sF_D)$.
\end{itemize}
 \end{cor}

 Second example is that,
 though the proof of Theorem \ref{T:thm1} uses Nash-type inequalities,
 it can  easily be replaced by using the following super-Poincar\'e inequalities
$$\|f\|_2^2\le r \sE(f,f)+\beta(r)
\|f\|_1^2, \quad f\in \sF, r>0,
$$
which obviously are simpler than
logarithmic Sobolev inequalities mentioned above.
See the proof of off-diagonal Dirichlet heat kernel upper bounds in
Subsection \ref{section5.1}.
The reader  is  referred to \cite{W} for more details on super-Poincar\'e inequalities.

\medskip

We now comment on the assumption of $\Lambda (\psi)^2<\infty$ in Theorem \ref{T:thm1}.
In some  contexts (for example, anomalous Markov processes such as
diffusions or continuous time random walks on fractals with walk dimension larger than 2),
the energy measure $\Gamma(u,  u)$ is singular with respect to the measure $m$ unless $u$ is a constant,
and so in this case
the off-diagonal upper bounds \eqref{e:1.5} in Theorem \ref{T:thm1}  degenerate into the on-diagonal upper bounds \eqref{e:ondup0}.
 Indeed, (i) implies (iii) in Theorem \ref{T:thm1} is a special case of a
 more general result,  Theorem \ref{T:3.2},
 of this paper. We emphasize that Theorem \ref{T:3.2} does not have this
absolute continuity assumption and is applicable to the singular contexts mentioned above.

We like to mention
 another beautiful result due to Grigor'yan \cite{Gri2}
that gives a  two-point  off-diagonal Gaussian upper bound
  of $p(t,x,y)$ for Brownian motions on a smooth connected manifold
from the on-diagonal upper bounds
at two points $x$ and $y$ only rather than at all
points.
The approach of \cite{Gri2} is based on
an integral maximum principle,  which is quite
different from the functional-inequalities adopted in the present paper (because the functional-inequalities technique
yield a uniform upper bound for heat kernel).
The approach of \cite{Gri2} has been extended
to discrete time random walks on graphs in \cite{CGZ}  and   to  continuous time random walks in \cite{F, C}.

 \medskip

We note that an
integral maximum principle was also established in \cite[Lemma 4.3]{Gri} via the mean-value type theorem. The advantage of this approach based on the integral maximum principle is that it can
handle space-inhomogeneous
heat kernel estimates; for example, the heat kernel estimate on a manifold $(M,\rho)$ of a non-negative Ricci curvature has the following upper bound
$$
p(t,x,y)\le \frac{C_\varepsilon}
{\sqrt{v(x,\sqrt{t})\, v(y,\sqrt{t})}}
\exp\left(-\frac{\rho(x,y)^2}{(4+\eps)t}\right),\quad x,y\in M, \varepsilon>0,
$$
where
$v(x,r)$
is the Riemannian volume of
the geodesic ball with the centre $x\in M$ and of the radius $r$. See \cite[Section 5]{Gri} for more details.
By slightly modifying the proof of Theorem \ref{T:thm1}),  we establish in Theorem \ref{T:4.1}
a possibly space-inhomogeneous  off-diagonal upper bound estimate
for Dirichlet heat kernel from its   possibly  space-dependent on-diagonal upper bound.
This result will then be used  in \cite{CKKW} to  study
two-sided heat kernel estimates for reflected diffusions with jumps on metric measure spaces
that satisfy
general volume doubling condition.

\medskip

The rest of this paper is organized as follows. In Section \ref{section2},
 we briefly describe the setting of our paper and give some definitions for regular properties of functions. In particular, we prove that
for any decreasing function
 $\varphi:\R_+\to \R_+$ such that $1/\varphi$ has the doubling property, there is $\bar\varphi\in \mathcal{R}$ which is
 comparable to
 $\varphi$ by neglecting constants. In Section \ref{Section3}, we mainly recall the results from \cite{Cou} about the characterizations of Nash-type inequalities and on-diagonal heat kernel bounds. Section \ref{section4} is the main part of our paper, and is devoted to off-diagonal heat kernel bounds. We introduce a new generalized Davies' method
 and obtain
 off-diagonal heat kernel bounds from Nash-type inequalities in
a very general setting.
As an application of our approach,
we give in Theorem \ref{T:4.1}
an off-diagonal Dirichlet heat kernel estimates from
its possibly space-dependent on-diagonal upper bound.
Several examples are given in Subsection \ref{S:5.2}
to illustrate the strength of our main results.

 \smallskip

\noindent
{\bf Notations:}
We use ``$:=$" as a way of  definition.
 For $a, b\in \R$, $a\wedge b:=\min \{a, b\}$ and $a\vee b:=\max\{a, b\}$.
 We denote $\R_+:=(0,\infty)$. For any function $f:\R_+\to \R_+$, $f(0):=\lim_{\to 0}f(r)$ and $f(\infty):=\lim_{r\to\infty}f(r)$.
For any open set $B\subset E$, denote  the set of continuous functions on $B$ with compact support by  $C_c(B)$.
For a function space $\sH$
on $E$, we let $\sH_b:=\{f\in \sH: f \mbox{ is bounded}\}$.
The constants $c$, $C$ and $c_i$, whose exact
values are unimportant, are changed in each statement and proof. In particular, we write $c:=c(C_0, C,\eta,\cdots)$ which means that the constant $c$ depends on $C_0$, $C$, $\eta$ and $\cdots$ only.

\section{Preliminaries}\label{section2}
\subsection{Dirichlet form and heat kernel}\label{section2.1}
In this subsection, we first describe the setting of this paper.
Let $(E, \rho)$ be a locally compact separable metric space,
and $m$ a $\sigma$-finite Radon measure on $E$ with full support.
 Suppose that $(\sE,\sF)$ is a regular symmetric Dirichlet
form on $L^2(E; m)$. This is,

\begin{description}
\item{(i)}  $\sF$ is a linear subspace of $L^2(E;m)$, and $\sE$ is a symmetric bilinear form on $\sF\times\sF $
so that $\sE(u, u) \geq 0$ for every $u\in \sF$;

\item{(ii)}   $\sF$ is a Hilbert space with inner product $ {\sE_1} (u, v):=\sE(u, v)+\int_E u(x) v(x) \, m(dx)$;

\item{(iii)}  (Markovian property)   for any $u\in \sF$, $v:= (0\vee u) \wedge 1\in \sF$ and
$\sE(v, v) \leq \sE(u, u)$;

\item{(iv)} (regularity)
$\sF\cap C_c(E)$  is dense both in $(\sF,  \sqrt{\sE_1})$ and in $(C_c(E),\|\cdot\|_\infty)$.
\end{description}

 We say the Dirichlet form $(\sE,\sF)$ is strongly local, if for any $u,v\in \sF$ with compact support such that $u$ is constant on a neighbourhood of the support of $v$, $\sE(u,v)=0.$

By the Beurling-Deny formula,
any regular symmetric  Dirichlet form $(\sE,\sF)$ admits the following representation for any $u,v\in \sF$,
\begin{align*}\sE(u,v)=&\sE^{(c)}(u,v)+\frac{1}{2}\iint_{E\times E \backslash {\rm diag}}(u(x)-u(y))(v(x)-v(y))\,J(dx,dy)+\int_E u(x)v(x)\,\kappa(dx).\end{align*} Here $\sE^{(c)}$ is a symmetric form possessing the strongly local
property, $J(dx,dy)$ is a symmetric
Radon measure on $E\times E \backslash {\rm diag}$, where ${\rm diag}$ denotes the diagonal set, and $\kappa(dx)$ is a
Radon measure on $E$.
See \cite[Theorem 4.5.2]{FOT} or \cite[(4.3.1)]{CF}.
 We call $J(dx, dy)$ and $\kappa(dx)$
 the jumping measure  and the killing measure of $(\sE, \sF)$, respectively.
 {\it Throughout this paper,
we assume the Dirichlet form $(\sE,\sF)$ admits no killing inside $E$; that is,  $\kappa(dx)=0$.}

\medskip

Let $(\sL , {\sD}(\sL))$
be the $L^2$-infinitesimal generator of $(\sE, \sF)$;
that is, $u\in \sD (\sL)$,
if and only if $u\in \sF$ and there is some function $f\in L^2 (E; m)$ so that
 $$
  \sE(u,v)=-\int_E f(x) v(x)\, m(dx)
 \quad\hbox{for every }  v\in \sF.
 $$
 In this case, we write $f$ as $\sL u$.
   It is known that there
  exist a Borel exceptional subset $\NN$ of $E$ and
  an  $m$-symmetric Hunt process $X=\{X_t, t\geq 0; \bP_x, x\in E\setminus \NN\}$
  on $E$ associated with the  regular symmetric Dirichlet form $(\sE,\sF)$ on $L^2(E; m)$ in the sense that the
  transition semigroup $(P_t)_{t\geq 0}$ defined by
  $$
  P_t f(x) = \bE_x [ f(X_t) ], \quad x\in E\setminus \NN, \  t\geq 0
  $$
  is a strongly continuous symmetric  semigroup on $L^2(E; m)$ whose infinitesimal generator
  is the  $L^2$-infinitesimal generator
   $(\sL,  \sD (\sL))$ of $(\sE, \sF)$ on $L^2(E; m)$.
 A set $\sN$ is said to be properly exceptional
 for the Hunt process $X$, if $m(\mathcal{N})=0$ and $\bP_x(X_t\in \mathcal{N}$ for some $t>0)=0$ for all $x\in E\backslash \mathcal{N}$. It is known that a subset
 has zero $\sE_1$-capacity, if and only if it is contained in a Borel properly exceptional set.
The Hunt process $X$ is unique up to a set of zero $\sE_1$-capacity.
 The transition semigroup $(P_t)_{t\geq 0}$ can be extended
  to be  a strongly continuous contraction semigroup in $L^p(E; m)$
  for every $p\in [1, \infty]$.
 Every function in $\sF$ admits a quasi-continuous version,
 and we will always use its quasi-continuous representation.
We refer the detailed theory of symmetric Dirichlet forms to \cite{CF, FOT}.

\medskip

\begin{defn}\rm  A
measurable function
$p(t, x,y): (0,\infty)\times (E \backslash \NN) \times (E \backslash \NN) \to [0,\infty)$
is called the  \emph{heat kernel} associated with the
transition semigroup
$(P_t)_{t\ge0}$, if $\NN$ is of zero $\sE_1$-capacity and
the following hold:
\begin{align*}
P_tf(x) &=  \int p(t, x,y) f(y) \, m(dy)  \quad \hbox{for all } x \in E \backslash \NN,
f \in L^1(E;m), \\
p(t, x,y) &= p(t, y,x) \quad \hbox{for all } t>0,\, x ,y \in E \backslash \NN, \\
p(s+t , x,y) &= \int p(s, x,z) p(t, z,y) \,m(dz)
\quad \hbox{for all } s,t>0,
\,\, x,y \in E \backslash \NN.
\end{align*}\end{defn}

It was proved in \cite[Theorem 3.1]{BBCK} that, if there is a positive left continuous function $\varphi(t)$ on $\R_+$ so that
 $$\|P_tf\|_\infty\le \varphi(t)\|f\|_1
 \quad \hbox{for } f\in L^1(E;m) \hbox{ and } t>0,
 $$
  then there exists a properly exceptional set $\NN$ such that the semigroup $(P_t)_{t\ge0}$ has
 a heat kernel $p(t, x,y)$ defined on $(0,\infty)\times (E \backslash \NN) \times (E \backslash \NN)$.  Note that we can always extend (without further
 mentioning) $p(t,x,y)$ to all $x,y\in E$ by setting $p(t,x,y)=0$ if
 $x$ or $y$ is in $ \NN$.
 So one can also extend the transition function
  $$P(t,x,dy)=p(t,x,y)\,m(dy)=\P_x(X_t \in dy)$$ of $X$ into $(0,\infty)\times E\times \mathcal{B}(E)$.

\medskip

Recall that $\sF_b$ is the collection of bounded functions in $\sF$.
It is well known that for any
$u\in \sF_b$, there exists a unique positive
Radon measure
$\Gamma(u,u)$ on $E$
so that
\begin{equation}\label{e:1.1}
\int_E v \, d\Gamma(u,u)=\sE(u, uv)-\frac 12\sE(u^2,v)
\quad \hbox{for every } v\in \sF_b.
\end{equation}
 Furthermore,  for any $u,v\in \sF_b$,
$\Gamma(u,v)$ is defined by polarization.
The Radon measure $\Gamma(u,u)$ can be uniquely extended to any $u\in \sF$
as the increasing limit of $\Gamma (u_n, u_n)$, where
$u_n:=((-n)\vee u)\wedge n$.
The measure $\Gamma(u,u)$ is called the \emph{energy measure} of $u$
for $(\sE, \sF)$.  See  \cite[Chapter 3.2]{FOT} or \cite[Chapter 4.3]{CF} for more details.
When $\Gamma(u,u)$ is absolutely continuous with respect to the measure $m$, its Radon-Nikodym
is also called  the {\it square field operator} or the {\it carr\'e du champ} in some literature.
Indeed,  we can define $d\Gamma(u,u)$,
for any $u\in \sF_b$, by the weak limit of the measure
 $$
 d\Gamma_t(u,u)=\frac{1}{2t}\left(
\int_{E} (u(x)-u(y))^2 P(t,x,dy)
 \right)\,m(dx).
 $$
 Then by the fact that $$(u(x)-u(y))(u(x)v(x)-u(y)v(y))-u(x)(u(x)-u(y))(v(x)-v(y))=(u(x)-u(y))^2v(y)$$ and the symmetry of $P(t,x,dy)\,m(dx)$, we can easily verify that \eqref{e:1.1} holds.
 Note that, since the Dirichlet form $(\sE,\sF)$ admits no killings inside $E$, for any $u\in \sF_b$,
 $$
 \lim_{t\to 0} \frac{1}{t}\int_E u(x)^2(1-P_t1(x))\,m(dx)=0.
 $$
 We take this opportunity to point out  that no killings inside $E$ condition needs to be imposed
 in \cite{CKS}; otherwise, \cite[(1.4) and (3.4)]{CKS} do not hold.
 Similar to the remark before \cite[Theorem (3.9)]{CKS}, we can extend the definition of $\Gamma(u,u)$ from $\sF_b$ into $\hat \sF_b=\{u+c: u\in \sF_b, \, c\in \R \}$
  by setting   $\Gamma(u+c, u+c)= \Gamma (u, u)$ for $u\in \sF_b$ and $c\in \R$.
 In particular, $\Gamma (e^\psi, e^\psi)$ is well defined for every $\psi \in \sF_b$
as $e^\psi -1 \in \sF_b$. The reader is  referred to \cite[Chapter 4.3]{CF} for the construction of the energy measure $\Gamma(u,u)$   by using  a probabilistic approach, and to \cite[Chapter 3.2]{FOT} using an analytic approach.

\begin{example}\rm Let $(M,\rho)$  be a manifold with Riemannian volume measure $\mu$,
and $\sE(u,u)=\frac{1}{2}\int_M |\nabla u|^2\sigma\,d\mu$, where $\sigma>0$ is a measurable function on $M$ with $\sigma,\sigma^{-1}\in L_{{\rm loc}}^\infty(M;\mu)$. Let $dm=\sigma\,d\mu$. Then, $d\Gamma(u,u)=\frac{1}{2}|\nabla u|^2\,dm.$
\qed\end{example}

\begin{example} \rm Let $(\sE,\sF)$ be a symmetric $\alpha$-stable-like Dirichlet form on $L^2(\R^d;dx)$ as follows
\begin{align*}
\sE(u,v)&=\frac{1}{2}\iint_{\bR^d\times \bR^d\setminus {\rm diag}}(u(x)-u(y))(v(x)-v(y))  \frac{c(x,y)}{|x-y|^{d+\alpha}} \,d x \,d y,\\
\sF &=\{u\in L^2(\bR^d;dx):\sE(u, u)<\infty\}, \end{align*} where $\alpha\in (0,2)$ and $c_0^{-1}\le c(x,y)=c(y,x)\le c_0$ for some $c_0\ge1$. Then,
$$
\Gamma(u,u) (dx) =\frac{1}{2} \left( \int_{\{y\in\R^d: y\neq x\}}(u(x)-u(y))^2  \frac{c(x,y)}{|x-y|^{d+\alpha}} \,d y \right) dx.
$$
\qed
\end{example}

\subsection{Doubling function, regular function and  condition (D)}\label{section2.2}

\begin{defn}\rm
An increasing function $f:\R_+\to\R_+$ is said to be a doubling function, if there is a constant $c>0$ such that
  $$
  f(2r)\le cf(r) \quad \hbox{for } r>0.
  $$
 The above  is equivalent to that
 there are constants $C\ge1$ and $\eta>0$ so that
$$\frac{f(R)}{f(r)}\le C\left(\frac{R}{r}\right)^\eta \quad \hbox{for any } 0<r\le R.
$$
In this case, we also say  that  $f$ has the doubling property.
 \end{defn}

\medskip

The following definition is taken from \cite[p.\ 515]{Cou}
(see also \cite[Definition 2.1]{Gri}).

\begin{defn} \rm  A differentiable function $f:
\R_+ \to \R_+$ is said to satisfy condition (D), if there exists some constant $\lambda>0$ such that, with $M(t):=-\log f(t)$,
\[
 M'(s)\ge \lambda M'(t) \quad \hbox{for any }  t>0 \hbox{ and }  s\in [t,2t].
\]
 \end{defn}

For instance, $C^1$-functions $f(t)$ that behave like $t^{-\delta_1}$ for small $t$ with $\delta_1>0$, and $e^{-c_0t^{\delta_2}}$ for large $t$ with $c_0,\delta_2>0$ satisfy condition (D).

\medskip

 We have
 the following property for  the doubling function.

\begin{lem}\label{L:3.11} Let $\varphi:
\R_+ \to \R_+$ be a decreasing function. Then $1/\varphi$ satisfies the  doubling property, if and only if there exist  a strictly decreasing
$C^1$ function $\bar \varphi:
\R_+ \to \R_+$,
and positive constants $c_1\le c_2$ and
$b_1\le b_2$, such that
the following hold for all $r>0${\rm:}
\begin{equation}\label{e:2.5-}
c_1\varphi(r)\le \bar \varphi(r)\le c_2 \varphi(r)\end{equation} and
\begin{equation}\label{e:2.5}
b_1\le
-\frac{\bar\varphi'(r)}{\bar\varphi(r)}r\le
b_2.
\end{equation}
In particular,
\begin{itemize}
\item [{\rm(i)}]
The function $\bar \varphi$ satisfies condition {\rm(D)}. In fact, with $M(r):=-\log \bar\varphi(r)$
$$
\frac{b_1}{r}\le M'(r)=-\frac{\bar
\varphi'(r)}{\bar \varphi(r)} \leq    \frac{b_2}{r},\quad r>0.
$$

\item[{\rm (ii)}] If $\varphi(0)=\infty$ and $\varphi(\infty)=0$, then there is a constant $c\ge1$ such that
\begin{equation}\label{e:lll}c^{-1}\frac{r}{\bar\varphi^{-1}(r)}\le
-\bar\varphi'(\bar\varphi^{-1}(r))\le
c \frac{r}{\bar\varphi^{-1}(r)},\quad r>0.\end{equation} \end{itemize} \end{lem}

\begin{proof}
Suppose that  \eqref{e:2.5} holds. Then,
 \[
\log (\bar\varphi(2r)/\bar\varphi(r))=\int_r^{2r}\frac{\bar\varphi'(u)}{\bar\varphi(u)}\,du\ge -b_2\int_r^{2r}\frac{1}{u}\,du
=-b_2\log 2.
\]
So $1/\bar\varphi(r)\leq 1/\bar\varphi(2r)\le 2^{b_2} /\bar\varphi(r)$. In particular, $1/\bar\varphi(r)$ is a doubling function. This yields that  $1/\varphi$ satisfies the  doubling property
 as well if \eqref{e:2.5-} holds.

Suppose that  $1/\varphi$ satisfies the  doubling property.  Let $\varphi_0(2^i)= \varphi(2^i)$
for all $i\in \mathbb Z$ and extend it linearly
on $\R_+$.
Then $\varphi_0$ is
continuous and
decreasing on $\R_+$.
For  any $r>0$, choose $i\in \mathbb Z$ such that  $2^i\le r < 2^{i+1}$.
Then,   by the doubling property of the function $1/\varphi$, we have
\begin{equation}\label{e:compar}
\frac 1{C_1} \varphi(r)\le
 \varphi(2^{i+1}) \le \varphi_0 (r)\le \varphi(2^{i})\le C_1 \varphi(r)
\end{equation}for some constant $C_1>1$ independent of $i$ and $r$.
In particular, $1/\varphi_0$ is a doubling function on $\R_+$. Now, we define
$$\bar\varphi(r):=\frac{1}{r}\int_r^{2r}\varphi_0(u)\,du,\quad r>0.$$ It is clear that $\varphi_0(2r)\le\bar\varphi(r)\le \varphi_0(r)$. This along with \eqref{e:compar} yields \eqref{e:2.5-}, and this together with the doubling property of $1/\varphi_0$ implies the doubling property of $1/ \bar\varphi$ as well.
Furthermore, noticing that $\varphi_0$ is continuous and decreasing on $\R_+$, for all $r>0$,
$$\bar\varphi'(r)=-\frac{1}{r^2}\int_r^{2r}\varphi_0(u)\,du+\frac{1}{r}(\varphi_0(2r)-\varphi_0(r))\ge -\frac{2\varphi_0(r)}{r}$$
 and $$\bar\varphi'(r)=-\frac{1}{r^2}\int_r^{2r}\varphi_0(u)\,du+\frac{1}{r}(\varphi_0(2r)-\varphi_0(r))\le -\frac{1}{r^2}\int_r^{2r}\varphi_0(u)\,du\le -\frac{\varphi_0(2r)}{r}
 \le -
 C_2\frac{\varphi_0(r)}{r}, $$ where  in the last inequality above $
 C_2>0$ is independent of $r$ and we used the doubling property of $1/\varphi_0$. Combining both inequalities above with \eqref{e:compar}, we know that \eqref{e:2.5} holds; in particular, $\bar\varphi(r)$ is strictly decreasing on $\R_+$.

 (i) directly follows from \eqref{e:2.5}, and (ii) is a consequence of \eqref{e:2.5-} and \eqref{e:2.5}. The proof is complete.
 \qed
\end{proof}

\begin{lem}\label{L2.5} Let $\varphi:
\R_+ \to \R_+$ be a deceasing function such that $1/\varphi$ has the doubling property. Then
for any $\lambda>2$, there is a constant $C(\lambda)>0$ such that
$$\prod_{k=1}^\infty\left(\varphi \big( (\lambda-1)\lambda ^{-(k+1)}t\big)\right)^{1/ 2^{k} }\le C(\lambda)\varphi(t),\quad t>0.$$  \end{lem}
\begin{proof} Since $1/\varphi$ is a doubling function on $\R_+$, there are constants $C,\eta>0$ such that
$$
\frac{\varphi(r)}{\varphi(R)}\le C\left(\frac{R}{r}\right)^\eta,\quad 0<r\le R.
$$
Then, for any $\lambda>2$ and $t>0$,
\begin{align*}
  \prod_{k=1}^\infty\left(\varphi \big( (\lambda-1)\lambda ^{-(k+1)}t\big)\right)^{1/ 2^{k} }
 &\le \prod_{k=1}^\infty\left(
 \varphi \big( (\lambda-1) \lambda^{-1} t\big)
  \cdot C  \lambda^{\eta k}  \right)^{1/ 2^{k} } \\
&= C \left(
  \varphi ((\lambda-1)\lambda^{-1}t)\right)^{\sum_{k=1}^\infty  2^{-k}  }
  \lambda^{ \sum_{k=1}^\infty  \eta k /2^k} \\
&\le C_1\varphi \big((\lambda-1)\lambda^{-1}t\big)\le C_2
  \varphi
(t),
\end{align*} where $C_1,C_2>0$ are independent of $t$ but depend on
$C$, $\eta$ and $\lambda$. \qed
\end{proof}

Combining Lemmas \ref{L:3.11} and \ref{L2.5} with Definition \ref{D1.1}, we have the following statement.
\begin{propn}\label{P:dou}Let $\varphi:
\R_+ \to \R_+$ be a decreasing function such that $\varphi(0)=\infty$, $\varphi(\infty)=0$ and that $1/\varphi$ has the doubling  property. Then  there exist  $\bar\varphi\in \mathcal{R}$ and a constant $c\ge1$ such that
$$c^{-1}\bar\varphi(r)\le \varphi(r)\le
c\bar\varphi(r),\quad r\in \R_+.$$ \end{propn}

\section{ On diagonal heat kernel upper bounds}\label{Section3}
In this section, we
summarize
the relations between Nash-type inequalities and on-diagonal heat kernel upper bounds from \cite{Cou}. We emphasize that the results of this section hold for all regular symmetric Dirichlet forms
which may have killings inside $E$.

\begin{propn}\label{T:2.1}
Suppose that $(P_t)_{t\geq 0}$ is a symmetric
Markov semigroup on $L^2(E; m)$ with a regular symmetric Dirichlet form $(\sE, \sF)$ on $L^2(E; m)$.
Let $\delta$ be a non-negative constant.
\begin{description}
\item{\rm (i)}   Let  $\theta: \R_+\to \R_+$
be
such that $\int^{+\infty} 1/\theta (s) \,ds <\infty$. If
\begin{equation}\label{e:nash--}
 \theta  (\|f\|_2^2)\leq
  \sE(f,f)+\delta\|f\|_2^2
\quad \hbox{for }f\in \sF \hbox{ with } \| f\|_1 \leq1,
\end{equation}
then
$$
\|P_t \|_{1\to\infty}\leq \varphi(t) e^{\delta t} \quad \hbox{for } t>0,
$$
where $\varphi(t)$ is the
inverse function of $t\mapsto\int^\infty_t 1/\theta (s)\, ds$.

\item{\rm (ii)} Suppose that
$$
\|P_t \|_{1\to\infty}\leq\varphi(t)
e^{\delta t} \quad \hbox{for } t>0,
$$
 where $\varphi(t) :\R_+\to \R_+$ is a decreasing function.
 Then
\begin{equation}\label{e:2.3}
\wt \theta  (\|f\|_2^2)\leq  \sE(f,f)
+\delta\|f\|_2^2
\quad \hbox{for }f\in \sF \hbox{ with } \| f\|_1\leq 1,
\end{equation}
where
\begin{equation}\label{e:2.4}
 \wt \theta (r):=\sup_{t>0} \frac{r}{ t}\log \frac{r}{\varphi(t)},
 \quad r>0.
\end{equation}
\end{description}
\end{propn}

\pf
For any $t>0$ and $f\in L^2(E;m)$,
define $$P_t^\delta f(x)=e^{-\delta t}P_tf(x),\quad x\in E.$$ Then, $(P_t^\delta)_{t\ge0}$ is a symmetric Markov semigroup on $L^2(E;m)$ associated with the Dirichlet form
$(\sE_\delta,\sF)$, where
 $$
 \sE_\delta(u,u) :=
\sE(u,u)+\delta\|u\|_2^2,\quad u\in \sF.$$ Hence, by applying  \cite[Propositions II.1 and II.2]{Cou} to $(P_t^\delta)_{t\ge0}$ and using the fact that $\|P_t\|_{1\to \infty}=e^{\delta t}\|P_t^\delta\|_{1\to \infty}$ for all $t>0$,
we get the desired conclusion.   \qed.

\begin{remark} \rm
\begin{description}
\item{(i)} In  \cite[Proposition II.1]{Cou}, the function $\theta$ is assumed to be continuous but in fact  this continuity assumption
is not needed; see the proof of
\cite[Proposition II.1]{Cou}.
If the Dirichlet form $(\sE, \sF)$ is conservative; that is, if $P_t 1=1$ $m$-a.e. for every $t>0$,
then \eqref{e:nash--} can be weakened to hold for all  $f\in \sF $  with  $\| f\|_1 =1$;  see \cite[Remark on p.\ 513]{Cou}.
Since
$\int^\infty 1/\theta(s)\,ds<\infty$,
we can well define $\varphi(0):=\infty$ for the function $\varphi$ given in Proposition \ref{T:2.1}(i).

\item{(ii)}
 Since $(P_t)_{t\ge0}$ is symmetric on $L^2(E;m)$, $\|P_t\|_{1\to2}^2=\|P_{2t}\|_{1\to \infty}$ for all $t>0$.
Note that $\wt \theta$ defined by \eqref{e:2.4} is convex,
because
it is a supremum of a class of
convex
functions and $\widetilde \theta(0)=0$. Hence,
 \eqref{e:2.3} is equivalent to
$$
\wt \theta  (\|f\|_2^2)\leq  \sE(f,f)+ \delta\|f\|_2^2
\quad \hbox{for }f\in \sF \hbox{ with } \| f\|_1= 1.
$$
We also note that, if $\varphi(\infty)>0$, then $\widetilde \theta(r)=0$ for all $r\in (0,\varphi(\infty)]$;
if $\varphi(0)<\infty$, then $\widetilde \theta(r)=\infty$ for all $r\in (\varphi(0),\infty)$.  \end{description}
\end{remark}

\medskip

 It is clear that,
when $\theta$ is continuous, the function $\varphi(t)$ in Proposition \ref{T:2.1}(i) satisfies that
$$-\varphi'(t)=\theta(\varphi(t)),\quad t>0.$$ In order to study the equivalent characterizations between Nash-type inequalities and on diagonal heat kernel upper bounds, for a given strictly
decreasing differentiable bijection $\varphi: \R_+\to \R_+$
(in particular, $\varphi(0)=\infty$ and $\varphi(\infty)=0$) we define $\theta $  by
\begin{equation}\label{e:2.qw-}
\theta(r)= -\varphi'(\varphi^{-1}(r)),\quad r>0.\end{equation}
The following is proved in \cite[Lemma II.3]{Cou}.

\begin{lem}\label{e:2.qw}
Let $\varphi $ be a
 strictly decreasing differentiable bijection of $\R_+$ satisfying {\rm(D)}.
Then  there is a constant $c>0$ such that for all $r>0$,
 $$\widetilde \theta(r)\ge c \ \theta(r),$$
where $\widetilde \theta(r)$ and $\theta(r)$ are defined by \eqref{e:2.4} and \eqref{e:2.qw-}, respectively. \end{lem}

Note  that  by \eqref{e:2.qw-},
\begin{equation}\label{e:integral}
t_2-t_1 =\int_{\varphi(t_2)}^{\varphi (t_1)} 1/\theta(s)\,ds
\quad \hbox{for any } 0<t_1<t_2<\infty.
\end{equation}
Since  $\varphi $ is  strictly decreasing on  $\R_+$ with
$\varphi(0)=\infty$ and $\varphi(\infty)=0$.
we have by \eqref{e:integral} that  $  \int_{\varphi(t)}^\infty 1/\theta(s)\,ds =t$ for $t>0$
 and  $\int_0^\infty1/\theta(s)\,ds=\infty.$
Combining Proposition \ref{T:2.1} with Lemma \ref{e:2.qw}, we have the following equivalence between Nash-type inequalities
 and on diagonal heat kernel upper bounds
 (see \cite[Theorem II.5]{Cou}).

\begin{thm}\label{T:NASH} Let $(P_t)_{t\geq 0}$ be a symmetric Markov semigroup associated  with a regular symmetric Dirichlet form $(\sE, \sF)$ on $L^2(E; m)$.
For a
strictly decreasing  differentiable bijection $\varphi$ of $\R_+$ satisfying condition {\rm(D)},
 define $\theta(x)$ by \eqref{e:2.qw-}.
 Let $\delta$ be a non-negative constant. Then
  the following conditions are equivalent:
\begin{itemize}
\item[{\rm(i)}]
There is a constant $c_1>0$ such that
$$ \|P_t \|_{1\to\infty}\leq \varphi(c_1t)
e^{\delta t} \quad \hbox{for } t>0.
$$
\item [{\rm(ii)}]
There is a constant $c_2>0$ such that
$$c_2\,\theta(\|f\|_2^2)\le  \sE( f, f)
+\delta\|f\|_2^2
\qquad \hbox{for }f\in \sF \hbox{ with } \| f\|_1
\leq 1.
$$
\end{itemize}
\end{thm}

\begin{remark}\rm According to Lemma \ref{L:3.11} (in particular, by \eqref{e:lll}), we can replace the function $\theta$
in \eqref{e:2.qw-} by
$$ \theta_*(r):=\frac{r}{\varphi^{-1}(r)},\quad r\in \R_+,
$$
when  $\varphi$ in Theorem \ref{T:NASH} has the property
 that $1/\varphi(r)$ is a doubling function on $\R_+$.  \end{remark}

The following table lists some typical regular functions $\varphi$ and
the corresponding
$\theta(r):=-\varphi'(\varphi^{-1}(r))$
(by neglecting constant  multiplies).
In the table, $\alpha>0$, $\beta\in \R$ and $\gamma\in (0,1]$.

 {\footnotesize
\begingroup
\setlength{\tabcolsep}{11pt}
\renewcommand{\arraystretch}{1.8}
 $$\hspace{-12cm}\begin{array}{|c|c|c|c|c|c|}  \hline
 \hspace{-.1cm} \varphi(r) \hspace{-.1cm} &\hspace{-.1cm} r^{-\tfrac1\alpha}  \hspace{-.1cm}&\hspace{-.1cm} r^{-\tfrac1\alpha}\log^\beta(2+r)  \hspace{-.1cm}&\hspace{-.1cm} r^{-\tfrac1\alpha}\log^\beta(2+\frac1r)  \hspace{-.1cm}&\hspace{-.1cm} r^{-\tfrac1\alpha} e^{-r^\gamma} \hspace{-.1cm}&\hspace{-.1cm} \log^\beta(1+\frac1r)\hspace{-.1cm} \\  \hline
\hspace{-.1cm}  \theta(r)   \hspace{-.1cm}&\hspace{-.1cm} r^{1+\tfrac{1}\alpha}  \hspace{-.1cm}&\hspace{-.1cm} r^{1+\tfrac{1}\alpha} \log^{-\tfrac{\beta}\alpha}(2+\frac1r)  \hspace{-.1cm}&\hspace{-.1cm} r^{1+\tfrac{1}\alpha} \log^{-\tfrac{\beta}\alpha}(2+r)  \hspace{-.1cm}&\hspace{-.1cm} r^{1+\tfrac{1}\alpha} \vee r \log^{1-\tfrac{1}\gamma}(2+\frac1r) \hspace{-.1cm}&\hspace{-.1cm} r^{1+\tfrac{1}\beta}{\mathbbm 1}_{\{0<r\le 1\}}+e^{r^{1/\beta}}{\mathbbm 1}_{\{r>1\}} \hspace{-.1cm}\\   \hline
   \end{array}\hspace{-12cm}$$
  \endgroup
  }

\section{Off-diagonal heat kernel upper bounds}\label{section4}

In this section, we will derive off-diagonal
upper bounds for $p(t, x, y)$
from Nash-type inequalities.
 Let $(P_t)_{t\geq 0}$ be a symmetric Markov semigroup on $L^2(E; m)$
associated with a regular symmetric Dirichlet form $(\sE, \sF)$
on $L^2(E; m)$  that admits no killings inside $E$.
 Following Davies' idea,
 in order to study off-diagonal heat kernel bounds for $(P_t)_{t\ge0}$,
  we  consider the  perturbed
semigroup $(P_t^\psi)_{t\ge0}$ defined by
\begin{equation}\label{e:PS}
P_t^\psi f(x):=e^{\psi(x)}   P_t(e^{-\psi}f)  (x),\quad t\ge0
\end{equation} for some nice function $\psi\in \sF_b$.

  \begin{defn}\rm
    We say a function  $\psi\in   \sF_b$
  is admissible with respect to $(\sE, \sF)$, if
  there exists a constant
   $C_0>0$ such  that
 for every  $p\in[1,\infty)$   and $f\in \sF_b$,
\begin{equation}\label{e:01-}
\sE(e^\psi f^{2p-1}, e^{-\psi} f)\ge  \frac{C_0}{p}   \sE(f^p,f^p)-
\phi(C_0, \psi, p)
\|f\|_{2p}^{2p},
\end{equation}
where
$(\psi,r)\mapsto \phi( C_0, \psi,r)$
is a non-negative function on $\sF_b\times [1,\infty)$ such that
\begin{equation}\label{e:phS}
 \phi(C_0, \psi,r)= \phi(C_0, -\psi,r),\quad \psi\in \sF_b,  \  r\ge1,
\end{equation}
and  $r\mapsto \phi(C_0, \psi,r)$ is
increasing and has the doubling property uniformly
with respect to
$\psi$ on $[1,\infty)$, i.e., there are constants
$C:=C(C_0)\ge1$ and $\eta:=\eta(C_0)>0$ such that for all $\psi\in \sF_b$,
\begin{equation}\label{e:constants}
\frac{\phi(C_0, \psi, R)}{\phi(C_0, \psi, r)}\le C\left(\frac{R}{r}\right)^\eta,\quad 1 \le r\le R.
\end{equation}
We denote the class of such bounded admissible functions $\psi$ by ${\mathcal A}(C_0;\phi,C,\eta)$.
 \end{defn}

\begin{remark}\label{R:4.2} \rm
Note that if \eqref{e:01-} holds for $C_0>0$ and
  $\phi (C_0, \psi, p)$,
  it holds for any $\wt C_0\in (0, C_0)$
  and
$  \phi (\wt   C_0, \psi, p):=\phi (C_0, \psi, p)$.
\end{remark}

\begin{example}\label{exm4.2} \rm
For a given function $\psi\in
\sF_b$ such that
$$\Lambda (\psi)^2:=\max\left\{\left\|\frac{d e^{-2\psi}\Gamma(e^{\psi}, \  \
e^\psi)}{d m}\right\|_\infty, \left\|\frac{d e^{2\psi}\Gamma(e^{-\psi}, e^{-\psi})}{d m}\right\|_\infty\right\}<\infty,$$
we have
$\psi\in \mathcal{A}(s; \phi, 5/(1-s),2)$ for any $s\in (0,1)$,
where \begin{equation}\label{e:gann}
\phi (s, \psi, r)
= \left(1+\frac{4r^2{\mathbbm 1}_{\{r>1\}}}{1-s}\right) \Lambda(\psi)^2, \quad r \ge 1.
\end{equation}
In particular, $\phi(s, \psi,1)=\Lambda(\psi)^2.$
Furthermore, if $(\sE,\sF)$ is a strongly local Dirichlet form, then $\Lambda (\psi)^2=d\Gamma(\psi,\psi)/dm$ and $\psi\in \mathcal{A}(1; \phi,1,1)$
with $\phi(1, \psi,r):=r\Lambda (\psi)^2$.

Indeed, it was proved in \cite[(3.10)]{CKS} that for all $f\in \sF_b$,
$$
\sE(e^\psi f , e^{-\psi} f)\ge     \sE(f^p,f^p)- \Lambda(\psi)^2\|f\|_{2 }^{2 }.
$$
This implies that
 \eqref{e:01-} holds for $p=1$ with $C_0=s$ and $\phi(s, \psi,1)=\Lambda(\psi)^2$
for any $s\in (0, 1]$.
On the other hand,
we know from
the end of the proof for \cite[Theorem (3.9)]{CKS} that for all $p\in [1,\infty)$ and
$f\in \sF_b$,
$$
 \sE(e^\psi f^{2p-1}, e^{-\psi} f)\ge \frac{2p-1}{p^2} \sE(f^p,f^p)-4\sE(f^p,f^p)^{1/2}\Lambda(\psi)\|f\|_{2p}^p-\Lambda(\psi)^2\|f\|_{2p}^{2p}.
$$
Since for any $a,b\ge 0$,  $p\ge1$ and  $s\in (0,1]$,
$$
\frac{2p-1}{p^2}a^2-4ab-b^2\ge \frac{s}{p}a^2-\left(1+\frac{4p^2}{(2-s)p-1}\right)b^2,
$$
we have
\begin{align} \sE(e^\psi f^{2p-1}, e^{-\psi} f)\ge &  \, \frac{s}{p}\sE(f^p,f^p)-\left(1+\frac{4p^2}{(2-s)p-1}\right)\Lambda(\psi)^2\|f\|_{2p}^{2p}\label{e:00001}\\
\ge& \, \frac{s}{p}\sE(f^p,f^p)-\left(1+\frac{4p^2}{1-s}\right)\Lambda(\psi)^2\|f\|_{2p}^{2p}\label{e:00002}
\end{align}
for all $p\in[1,\infty)$, $s\in (0, 1)$
and
$f\in \sF_b$.
(Note that, by taking $s=1/2$ in \eqref{e:00001} and using $p\ge1$, we get that for all $p\in [1,\infty)$ and
$f\in \sF_b$,
$$
\sE(e^\psi f^{2p-1}, e^{-\psi} f)\ge (2p)^{-1} \sE(f^p,f^p)-9p \Lambda (\psi)^2\|f\|_{2p}^{2p},
$$ which corrects a typo in \cite[(3.11)]{CKS}; that is,  it should be  $(2p)^{-1}$ not $p^{-1}$ in front of $\sE(f^p,f^p)$.)
In particular, \eqref{e:00002} implies that, for
   $p>1$,
 \eqref{e:01-} holds with $C_0=s$ and $\phi(s, \psi,p)= \left(1+\frac{4p^2}{1-s}\right)\Lambda(\psi)^2$ for any $s\in (0,1)$.
Thus for any $1<  r\le R$ and $s\in (0,1)$,
$$\frac{\phi(s, \psi,R)}{\phi(s, \psi,r)}\le \frac{8R^2/(1-s)}{4r^2/(1-s)}\le 2(R/r)^2.
$$
Recall that, for $p=1$,
we have \eqref{e:gann} that
$\phi(s, \psi,1)=\Lambda(\psi)^2$ for $s\in (0, 1]$. Thus for any $r>1$ and $s\in (0,1)$,
$$ \frac{\phi(s, \psi,r)}{\phi(s, \psi,1)}\le  1+\frac{4r^2}{1-s} < \frac{5}{1-s}r^2.$$
  We conclude from these  two estimates   that  \eqref{e:constants} holds with $C=5/(1-s)$ and $\eta=2$.
  In other words,    $\psi\in \mathcal{A}(s; \phi, 5/(1-s),2)$ with $\phi(s, \psi,r)$ given by \eqref{e:gann}.

For a strongly
local Dirichlet form $(\sE,\sF)$, that $\Lambda (\psi)^2=d\Gamma(\psi,\psi)/dm$ immediately follows from the Leibniz rule. By the Leibniz rule again and the Cauchy-Schwarz inequality, for all $p\ge 1$, we have
$$
\sE(e^\psi f^{2p-1}, e^{-\psi} f)\ge p^{-1} \sE(f^p,f^p)-p \Lambda (\psi)^2\|f\|_{2p}^{2p};
$$ see e.g.\ \cite[(3.2)]{CKS}.
This yields that  $\psi\in \mathcal{A}(1; \phi,1,1)$, where $\phi(1, \psi,r)
=r\Lambda (\psi)^2$.
\qed
\end{example}

Now, we present  the main result of this paper.

\begin{thm}\label{T:3.2}
Let $(P_t)_{t\geq 0}$ be a symmetric Markov semigroup on $L^2(E; m)$
associated with
a regular symmetric Dirichlet form $(\sE, \sF)$
on $L^2(E; m)$ that has no killings inside $E$.
Suppose  the following Nash-type inequality holds{\rm\,:}
\begin{equation}\label{e:nash----}
 \theta  (\|f\|_2^2)\leq \sE(f,f)
 +\delta \|f\|_2^2
\quad \hbox{for }f\in \sF \hbox{ with } \| f\|_1
\leq1,
\end{equation}
 where
$\delta\ge0$, and $\theta:\R_+\to \R_+$ satisfies that
\begin{itemize}
\item[{\rm(i)}] $\int^\infty 1/\theta(s)\,ds<\infty\,;$
\item[{\rm (ii)}] there are constant
$c_1,
c_2\ge1$ and
an increasing function $\theta_*$ on $\R_+$ so that  $ \theta_*(r)/r$ is increasing and
\begin{equation}\label{e:4.8}
c_1^{-1}\theta_*(c_2^{-1}r)\le \theta(r)\le c_1\theta_*(c_2r)  \quad \hbox{for all } r\in \R_+ \, ;
\end{equation}

\item[{\rm(iii)}] for any $\lambda>2$ large enough,  there exist positive constants $c(\lambda)$ and $C(\lambda)$ so that
\begin{equation}\label{e:con-hk}
\prod_{k=1}^\infty \left(\varphi\big((\lambda-1)\lambda^{-(k+1)}t\big)\right)^{2^{-k}}\le C(\lambda)\varphi\big(c(\lambda)t\big),\quad t>0,
\end{equation}
where   $\varphi(t)$ is the
inverse function of $t\mapsto\int^\infty_t 1/\theta (s)\, ds$.
\end{itemize}
Then there are
a properly
exceptional set ${\cal N}\subset E$
and  a heat kernel
$p(t,x,y)$
associated with the semigroup $(P_t)_{t\ge0}$ and defined on $(0,\infty)\times (E\setminus {\cal N})\times (E\setminus {\cal N})$
so that
for any $\varepsilon>0$ and $\psi\in {\mathcal A}(C_0;\phi, C,\eta)$,
\begin{equation}\label{e:offdmaineq}
 p(t,x,y)\leq C_\varepsilon\, \varphi( c_\varepsilon t)
e^{ \delta  t} \;\exp\left( -|\psi(y)-\psi(x)|+
(1+\varepsilon)   \phi(C_0, \psi,1)     t \right)
\end{equation}
for all $t>0$ and $x,y\in E\setminus \NN,$ where
  $C_\varepsilon:=C_\varepsilon(\theta, C,\eta, \eps)>0$, $c_\varepsilon:=c_\varepsilon(\theta, C,\eta, \eps, C_0)>0$, and
 the dependencies  of $C_\varepsilon$ and $c_\varepsilon$ on $\theta$ are only through
 the constants $c_1$, $c_2$, $c(\lambda)$ and $C(\lambda)$
   in  \eqref{e:4.8} and \eqref{e:con-hk}.
 \end{thm}

We note that by the condition (ii) of the function $\theta$ it holds that $\int_0^\infty1/\theta(s)\,ds=\infty$. In particular, the function $\varphi$ in (iii) is strictly positive and strictly decreasing
on $\R_+$ with $\varphi(0)=\infty$ and $\varphi(\infty)=0$.

\medskip

To prove Theorem \ref{T:3.2}, we need the following
estimate
for $\|P_t^\psi\|_{1\to\infty}$, where $(P_t^\psi)_{t\ge0}$ is the perturbed semigroup defined by \eqref{e:PS}.

\begin{propn}\label{T:3.1}
Under the setting of Theorem $\ref{T:3.2}$,   for any $\varepsilon>0$ and $\psi\in {\mathcal A}(C_0;\phi, C,\eta)$,
 \bee\label{e:mar25-2}
\|P_t^\psi\|_{1\to\infty}\le
 C_\varepsilon\, \varphi( c_\varepsilon  t) e^{ \delta  t} \,  \exp\left(
   (1+\varepsilon)  \phi(C_0, \psi,1) t
   \right)
\quad \hbox{for all }  t>0,
 \eee
where
 $C_\varepsilon:=C_\varepsilon(\theta, C,\eta, \eps)>0$, $c_\varepsilon:=c_\varepsilon(\theta, C,\eta, \eps, C_0)>0$, and
 the dependencies  of $C_\varepsilon$ and $c_\varepsilon$ on $\theta$ are only through
 the constants $c_1$, $c_2$, $c(\lambda)$ and $C(\lambda)$
 in \eqref{e:4.8} and \eqref{e:con-hk}.
\end{propn}

\pf
 Without loss of generality, we assume the constant $C_0$ is less than $1$ (see Remark \ref{R:4.2}).
 We first assume that the function $r\mapsto \theta(r)/r$ is increasing on $\R_+$.
For any $\psi\in {\mathcal A}(C_0;\phi,C,\eta)$, non-negative
$f\in \sF_b$,
 $t>0$ and $p\ge1$, we write
\begin{equation}\label{e:1.3}
f_{t}(x):=P^\psi_{t} f(x)= e^{\psi(x)} \big( P_t (e^{-\psi}f) \big) (x)
\end{equation}
and
\begin{equation}\label{e:1.4}
f_{p,t}(x)
 := \exp \left(-
   \left( \phi(C_0, \psi,p)+
    \delta p^{-1} \right)  t \right) f_t(x) .
\end{equation}
For simplicity, in the following we write $\phi(\psi,p)$ for $\phi(C_0, \psi,p)$.
Note that
 since $e^\psi -1 \in \sF_b$,
$f_t \in \sF_b $ and so is $f_{p,t}$.
According to \eqref{e:01-} and
 \eqref{e:nash----},
 for any $p\ge1$ and $t>0$,
  \begin{align}
  \frac{d \|f_{p,t}\|_{2p}^{2p}}{dt}= & -
(2p\phi(\psi,p)+2 \delta)\|f_{p,t}\|_{2p}^{2p}-2p \sE(e^{\psi} (f_{p,t})^{2p-1}, e^{-\psi} f_{p,t})\nonumber\\
\le&
-2
 C_0\sE(f_{p,t}^p,f_{p,t}^p)-2   \delta \|f_{p,t}\|_{2p}^{2p}\label{e:3.5_1}\\
\le&
-2 C_0
   \big(\sE(f_{p,t}^p,f_{p,t}^p)+ \delta \|f_{p,t}\|_{2p}^{2p}\big)\nonumber \\
 =&
-2 C_0  \|f_{p,t}\|_p^{2p}\big(\sE(f_{p,t}^p/ \|f_{p,t}\|_p^{p},f_{p,t}^p/ \|f_{p,t}\|_p^{p})+\delta \|f_{p,t}\|_{2p}^{2p}/\|f_{p,t}\|_p^{2p}\big)\nonumber\\
\le&
-2
C_0  \|f_{p,t}\|_p^{2p}\theta(\|f_{p,t}\|_{2p}^{2p}/\|f_{p,t}\|_p^{2p}). \label{e:3.5}
\end{align}
 From  \eqref{e:3.5_1}, we see that $t\mapsto e^{2   \delta t}\|f_{p,t}\|_{2p}^{2p} $ is decreasing
on  $\R_+$ for every $p\geq 1$.
On the other hand, by \eqref{e:1.4},
\begin{align*}
 \|f_{2p,t}\|_{2p}^{2p}=&  \exp \left(-
 2p\phi(\psi, 2p)t-    \delta t \right)    \|f_{t}\|_{2p}^{2p} \\
  =&  \exp \left( -   2p\phi(\psi, 2p)t+2p\phi(\psi, p)t +    \delta t \right) \| f_{p,t}\|_{2p}^{2p} \\
  =&
        \exp \left( -   [2p(\phi(\psi, 2p)-\phi(\psi, p))+    \delta  ]t \right)(e^{2    \delta t} \| f_{p,t}\|_{2p}^{2p} ).
\end{align*}
  Since $\phi(\psi, 2p) \geq \phi(\psi, p)$,  it follows that
    $\|f_{ 2p,t}\|_{ 2p}^{2 p}$ is    decreasing in $t\in  \R_+$ for every $p\geq 1$.
    In other words,  $ \|f_{  p,t}\|_{  p}^{  p}$ is    decreasing in $t\in  \R_+$ for every $p\geq 2$.
     Hence we have from \eqref{e:3.5}
      and the increasing property of the function $r\mapsto
     \theta(r)/r$ on $\R_+$ that  for any $p\ge2$ and $t_0>0$,
\begin{equation}\label{e:1.5a}
\frac{d \|f_{p,t}\|_{2p}^{2p}}{dt}\le - 2 C_0\|f_{p,t_0}\|_p^{2p} \,
\theta \left(\|f_{p,t}\|_{2p}^{2p}/\|f_{p,t_0}\|_p^{2p} \right)\quad \hbox{for } t\ge t_0.
\end{equation}

Define
$$
\Psi(r)=  \frac{1}{2  C_0}\int_r^\infty \frac{ds}{\theta(s)},
 \quad r>0,
$$
which is finite by our condition {\rm(i)}.  We
can rewrite \eqref{e:1.5a}
as
$$ \frac{d }{ dt }\Psi\left(\frac{\|f_{p,t}\|_{2p}^{2p}}{\|f_{p,t_0}\|_p^{2p}}\right)
\ge 1 \quad \hbox{for } t\ge t_0.
$$
 In particular, for all $p\ge2$ and $t\ge t_0>0$,
$$ \Psi\left(\frac{\|f_{p,t}\|_{2p}^{2p}}{\|f_{p,t_0}\|_p^{2p}}\right) \ge  t-t_0 .
$$
That is,
$$
\|f_{p,t}\|_{2p}\le \left( \Psi^{-1} (t-t_0)\right)^{1/(2p)} \|f_{p,t_0}\|_p
=\left( \varphi(2  C_0 (t-t_0))\right)^{1/(2p)} \|f_{p,t_0}\|_p,\quad t\ge t_0.
$$
Consequently,  for all $p\ge2$
\begin{equation}\label{e:3.8}
\|f_{t}\|_{2p}\le \left(  \varphi(2 C_0 (t-t_0))\right)^{1/(2p)}
  \exp  \left(
  (  \phi(\psi,p)+   \delta p^{-1}) (t-t_0) \right)    \|f_{t_0}\|_p,\quad t\ge t_0.
\end{equation}

 Fix $t>0$, and,  for $k\ge1$, let $t_k=(\lambda-1)t \sum_{i=1}^{k}
  \lambda^{- i}=t(1-\lambda^{-k})$,
  where    $\lambda> 2^{\eta }$
  is  a large constant satisfying  \eqref{e:con-hk} and
 independent of $t$  to be determined later.  Applying
  \eqref{e:3.8} with  $p=
 2^k$, $t=t_{k+1}$ and $t_0=
 t_k$, we get
\begin{equation}
\label{e:newJ1---}
\|f_{t_{k+1}}\|_{2^{k+1}}\le \!\! \Big(
\varphi(2  C_0 (\lambda-1)\lambda^{-(k+1)}t)\Big)^{1/2^{k+1}}\!\!\!
  \exp \Big(  (\phi(\psi,2^k)+    \delta2^{-k})(\lambda-1)\lambda^{-(k+1)}t\Big)
\|f_{t_k}\|_{2^{k}}.
\end{equation}
Since $t\mapsto \|f_{p,t}\|_{p}$ is decreasing
on $(0,\infty)$ for all $p \ge 2$, $\|f_{2^k,t}\|_{2^{k}}\le \|f_{2^k, t_{k}}\|_{2^{k}}$ for $ k \ge 1$.
We have
by
 \eqref{e:1.4} that  for all $k\ge 1$,
\begin{equation}
\label{e:newJ1}\begin{split}
\|f_t\|_{2^{k}}&=\exp\left( ( \phi( \psi,2^k)+  \delta 2^{-k}) t\right) \|f_{2^k, t}\|_{2^{k}}\\
&\leq \exp\left( ( \phi( \psi,2^k)+  \delta 2^{-k}) t\right) \|f_{2^k, t_k}\|_{2^{k}}
\\
&=   \|f_{t_k}\|_{2^{k}}\exp\left( ( \phi( \psi,2^k)+  \delta 2^{-k})  (t-t_k)\right)\\
    & = \|f_{t_k}\|_{2^{k}}\exp\left( t ( \phi( \psi,2^k)+
    \delta 2^{-k}) \lambda^{-k} \right).\end{split}
\end{equation}

 As $\psi\in {\mathcal A}(C_0;\phi,C,\eta)$,
 we have
 \begin{align}
\label{e:newJ2}
  \frac{\phi(\psi,R)}{\phi(\psi,r)}\le C\left(\frac{R}{r}\right)^\eta  \quad \hbox{for } R\geq r\geq 1  .
\end{align}
Thus, using our assumption $\lambda> 2^{\eta }$,  \begin{align}
\label{e:newJ3}\limsup_{k\to\infty}  \big( \phi(\psi,2^k)+
    \delta 2^{-k}\big)  \lambda^{-k} \le \limsup_{k\to\infty}  \big( C (2^{\eta}/\lambda)^k\phi( \psi,1)+ \delta (2 \lambda)^{-k}\big) =0.\end{align}  \eqref{e:newJ1} and \eqref{e:newJ3}  along with the inequality \eqref{e:newJ1---} yield
\begin{equation}\label{e:1.10-}
\begin{split}
\|f_t\|_\infty =& \lim_{k\to\infty} \|f_t\|_{2^{k}} \le \limsup_{k\to \infty} \|f_{t_{k+1}}\|_{2^{k+1}}\\
 \le & \left(\prod_{k=1}^\infty \left(   \varphi\big(2 C_0
(\lambda-1)
\lambda^{-(k+1)}t\big)\right)^{1/ 2^{k+1}} \right)\\
&\times
 \exp \left(
 \lambda^{-1}(\lambda-1)t\left(\sum_{k=1}^\infty
 \phi(\psi,2^k)\lambda^{-k}+    \delta\sum_{k=1}^\infty 2^{-k}\lambda^{-k}\right) \right)  \|f_{(\lambda-1) \lambda^{-1}t}\|_{2}.
 \end{split}
\end{equation}
 Furthermore,
since $\|f_{1, s}\|_2$ is decreasing in $s\in [0, \infty)$,
$\| f_{1, s}\|_2 \leq \|f_{1, 0} \|_2=\| f\|_2$;
 that is,
$$
 \| f_s\|_2 \leq \exp \big(   (\phi(\psi,1)+    \delta)s\big) \| f\|_2,\quad  s>0.
$$
This together  with \eqref{e:1.10-}
gives us
\begin{equation}\label{e:1.10}
\begin{split}\|f_t\|_\infty \le & \left(\prod_{k=1}^\infty \left(
\varphi\big(2  C_0 (\lambda-1)
\lambda^{-(k+1)}t\big)\right)^{1/ 2^{k+1}} \right)
 \\ &\times
\exp \left(\lambda^{-1}
(\lambda-1)t\left(\sum_{k=0}^\infty
\phi(\psi,2^k)\lambda^{-k}+    \delta\sum_{k=0}^\infty 2^{-k}\lambda^{-k}\right)\right) \| f\|_2.
 \end{split}
\end{equation}

Using \eqref{e:newJ2}, we have that,  for
 $\lambda> 2^{\eta }$,
\begin{align*}
 \sum_{k=0}^\infty \phi(\psi,2^k)\lambda^{-k}
 \le \phi(\psi,1)+C \phi(\psi,1)\sum_{k=1}^\infty
 2^{\eta k}\lambda^{-k}
={\phi(\psi,1)}\Big(1+\frac{C 2^{\eta }}{\lambda -2^{\eta }}\Big).
 \end{align*}
Clearly,
$$    \delta\sum_{k=0}^\infty 2^{-k}\lambda^{-k}=\frac{   \delta}{1-2^{-1}\lambda^{-1}}.$$
These together with  \eqref{e:con-hk} and \eqref{e:1.10}  yield  that
\begin{align*}
\|f_t\|_\infty \le &
  \left( C(\lambda)\varphi(2 C_0 c(\lambda)  t)\right)^{1/2}
   \exp\left(
    t \,
     \frac{\lambda-1}\lambda \left( \phi(\psi,1) \Big(1+\frac{C 2^{\eta }}{\lambda -2^{\eta }}\Big)+\frac{    \delta}{1-2^{-1}\lambda^{-1}} \right)
  \right) \|f\|_2\\
  \le&  \left( C(\lambda)\varphi(2 C_0 c(\lambda) t)\right)^{1/2}
  \exp\left(
  t \phi(\psi,1)
    \frac{\lambda-1}\lambda\Big(1+\frac{C 2^{\eta }}{\lambda -2^{\eta }}\Big)+     \delta t
\right)\|f\|_2 .
  \end{align*}
   By the duality and \eqref{e:phS} (see \cite[p.\ 271]{CKS} for details), we get
$$
\|f_t\|_2\le    \left( C(\lambda)\varphi(2 C_0 c(\lambda) t)\right)^{1/2}
  \exp\left(
   t  \phi(\psi,1)
    \frac{\lambda-1}\lambda\Big(1+\frac{C 2^{\eta }}{\lambda -2^{\eta }}\Big)+     \delta t
\right) \|f\|_1.
$$
Since $P^\psi_t f= P^\psi_{t/2} (P^\psi_{t/2}f)$,  it follows that  for every non-negative $f\in \FF$,
\begin{equation}\label{e:kkggll}
\|f_t\|_\infty \le   C(\lambda)\varphi(  C_0 c(\lambda)
 t)
\exp\left( t
  \phi(\psi,1)
    \frac{\lambda-1}\lambda\Big(1+\frac{C 2^{\eta }}{\lambda -2^{\eta }}\Big)+   \delta t
\right) \|f\|_1.
\end{equation}
 For any $\varepsilon>0$,
taking $\lambda:=\lambda(\varepsilon,C,\eta)>
2^{\eta }$
sufficiently large such that
$$
\frac{\lambda-1}\lambda\Big(1+\frac{C 2^{\eta }}{\lambda -2^{\eta }}\Big) <1+ \eps,
$$
we conclude  that
with $C_\varepsilon:=C(\lambda)$ and
$c_\varepsilon:=C_0c(\lambda)$
(noting
  that  the constant $\lambda$  depends  on $\varepsilon$, $C$ and $\eta$ only),
$$
\| P^\psi_tf\|_\infty = \|f_t\|_\infty
\le C_\varepsilon
   \varphi(      c_\varepsilon  t) e^{     \delta t}
 \exp\left( (1+ \varepsilon )
 \phi(\psi,1)
 t \right)  \|f\|_1 \quad \hbox{for non-negative }
f\in \sF_b.
$$
The above inequality holds for general $f\in \FF_b$
by applying it  to $f_+$ and $f_-$ and then using
the triangular inequality.
 Therefore, we get the desired conclusion.

Now we consider the general case of
 $\theta$ satisfying conditions (i) and (ii). By condition (ii),
there are constants $c_1,c_2\ge1$ and an increasing function $\theta_*$ so that
 $ \theta_*(r)/r$ is increasing  and
 $$
 c_1^{-1}\theta_*(c_2^{-1}r)\le \theta(r)\le c_1\theta_*(c_2r)
 \quad \hbox{for all }   r\in \bR_+.
 $$
 Then
 \begin{equation}\label{e:4.17}
 c_2^{-2}\varphi_*((c_1c_2)^2r)\le \varphi(r)\le \varphi_*(r)  \quad \hbox{for  all }r\in \bR_+,
\end{equation}
 where $\varphi_*(r)$ is  the
 inverse
 function of $r\mapsto \int_r^\infty c_1/\theta_*(c_2^{-1}s)\,ds$.
The
function $c_1^{-1}\theta_*(c_2^{-1}r)$ in place of $\theta(r)$ satisfies conditions (i), (ii) and (iii) in Theorem \ref{T:3.2}.
Applying the above proof to $c_1^{-1}\theta_*(c_2^{-1}r)$ in place of $\theta(r)$ and noting \eqref{e:4.17},
we get the desired  assertion.
\qed

\medskip

We are now in the position to present the proofs for Theorem \ref{T:3.2}, Theorem \ref{T:thm1} and Corollary \ref{C:local}.

\medskip

\noindent
{\bf Proof of Theorem \ref{T:3.2}.}\, First, by
\eqref{e:nash----} and Proposition \ref{T:2.1}(i), we know that
$$\|P_t\|_{1\to\infty}\le \varphi(t)e^{\delta t},\quad t>0,$$ which along with   \cite[Theorem 3.1]{BBCK} yields that
there exist   a properly
exceptional set ${\cal N}\subset E$
and
 a heat kernel
 $p(t,x,y)$ on $(0,\infty)\times (E\setminus {\cal N})\times (E\setminus {\cal N})$  associated with the semigroup $(P_t)_{t\ge0}$
that satisfies
 $$
p(t,x,y)\le \varphi(t)e^{\delta t},\quad t>0, \,\, x,y\in E\setminus {\cal N}.$$
 Thus   we have by Proposition \ref{T:3.1} that for any $\varepsilon\in (0,1]$ and $\psi\in \mathcal{A}(C_0;\phi, C,\eta)$,
$$
e^{\psi(x)}p(t,x,y)e^{-\psi(y)}\le C_\varepsilon
 \varphi(      c_\varepsilon t) e^{    \delta t}
\exp\left((1+\eps)
\phi(C_0, \psi,1) t\right),
\quad t>0, \,\
 x,y\in E\setminus {\cal N},
$$
so the desired estimate follows.  \qed

 \begin{remark}\rm
 \begin{itemize}
\item[(i)]
Obviously, \eqref{e:nash----} implies that for any $\vartheta\in (0,1]$,
 \begin{align}
 \label{e:nsw}
 \vartheta\theta  (\|f\|_2^2)\le \sE(f,f)
 +\vartheta\delta \|f\|_2^2,\quad f\in \sF\hbox{ and } \|f\|_1\le 1.
  \end{align}
  Then, by Proposition \ref{T:2.1}(i), for any $t>0$,
 $\|P_t\|_{1\to\infty} \le \varphi(\vartheta t)e^{\vartheta \delta t}.$
Thus, the term $ \varphi( c_\varepsilon t)
e^{ \delta  t}$ in \eqref{e:1.5--}, \eqref{e:1.6},  \eqref{e:offdmaineq} and \eqref{e:mar25-2} can be replaced by
 $\varphi(\vartheta c_\varepsilon t)
e^{\vartheta \delta  t}$.
Note that,  \cite[Theorem 3.25]{CKS} is stated in that way with $\vartheta=\eps$.
We stated our results without introducing $\vartheta\in (0,1)$ because
it can be easily added in whenever needed through the inequality  \eqref{e:nsw},
and because the bound $\varphi(\vartheta c_\varepsilon t)
e^{\vartheta \delta  t}$ does not seem to be better than $ \varphi( c_\varepsilon t) e^{ \delta  t}$.
Note that for $\vartheta \in (0, 1)$,
   \eqref{e:nsw}  is a weaker inequality than  \eqref{e:nash----}.

\item[(ii)]
  For general $\varphi$ that satisfies \eqref{e:con-hk},
by the deceasing property of $\varphi$ on $\R_+$ and $\varphi(0)=\infty$, we can assume that  $c(\lambda)$ in \eqref{e:con-hk} is decreasing with respect to $\lambda$ and $C(\lambda)$ is increasing with respect to $\lambda$ such that
 $c(\infty)=0$
or $C(\infty)=\infty$. Thus, according to the proofs of Proposition \ref{T:3.1} and Theorem \ref{T:3.2}, we know that for the constants $c_\varepsilon$  and
$C_\eps$ in \eqref{e:mar25-2} either $\lim_{\eps\to0} c_\eps =0$ or $\lim_{\eps\to 0}C_\eps=\infty$.
 In particular, if $1/\varphi$ has the doubling property, then, by Lemma \ref{L2.5}, we can take $ c_\eps \equiv c>0$ and
$\lim_{\eps\to 0}C_\eps=\infty$.
\end{itemize}
\end{remark}

\noindent {\bf Proofs of Theorem \ref{T:thm1} and Corollary \ref{C:local}.} \,
The equivalence between (i) and (ii) has  been proved in Theorem \ref{T:NASH},
and (iii) implies (ii). So  we only need to show
that (i) implies (iii). Let $\varphi\in \mathcal{R}$, and $\theta(r)=-\varphi'(\varphi^{-1}(r))$ given in (i).
Recall that $t=\int_{\varphi(t)}^\infty 1/\theta(s)\,ds$ for all $t\in \R_+$ and $\varphi(\infty)=0$.
By property (ii) for the regular function,
there is a decreasing function $N(r): \R_+\to \R_+$ so that
$$
c_1 ^{-1} N(r)\le -\frac{\varphi '(r)}{\varphi (r)} \leq c_1 N(r)
 \quad \hbox{for  } r>0.
 $$
In other words,
 \begin{align}
\label{e:Nth}
c_1^{-1} N(\varphi^{-1}(r))\le \frac{\theta(r)}{r}\le c_1 N(\varphi^{-1}(r))  \quad \hbox{for  } r>0 .
\end{align}
Clearly $N(\varphi ^{-1} (r))$ is a decreasing function. So \eqref{e:4.8} holds with $\theta_*(r):=
r N(\varphi^{-1}(r))$.
 Property (iii) for the regular function is just \eqref{e:con-hk}. Thus by Theorem \ref{T:3.2}, there are
a properly
exceptional set ${\cal N}\subset E$
and
   a heat kernel
$p(t,x,y)$ associated with $(P_t)_{t\ge0}$ and defined on $(0,\infty)\times (E\setminus {\cal N})\times (E\setminus {\cal N})$ such that for any $\varepsilon>0$ and
$\psi\in {\mathcal A}(C_0;\phi, C,\eta)$,
there are constants $c_\varepsilon, C_\varepsilon>0$
 so that
\begin{equation}\label{e:4.19}
p(t,x,y)\leq C_\varepsilon\,
 \varphi(      c_\varepsilon t) e^{     \delta t}
\;\exp\left( -|\psi(y)-\psi(x)|+
(1+\varepsilon)  \phi(C_0,\psi,1)     t \right)
\end{equation}
 for all $t>0$,
   and  $x,y\in E\setminus \NN$.

On the other hand, for any  $\psi\in
\sF_b$ such that
$\Lambda (\psi)^2<\infty,$ by Example \ref{exm4.2}, we know that
$\psi\in \mathcal{A}(s; \phi, 5/(1-s),2)$ for any $s\in (0,1)$,
where
$$
\phi(s, \psi,r)=\left(1+\frac{4r^2{{\mathbbm 1}}_{\{r>1\}}}{1-s}\right)\Lambda(\psi)^2.
$$

 Noting that $\phi(1/2, \psi,1)=\Lambda(\psi)^2$, we have from the above two displays that
$$
p(t,x,y)\leq C_\varepsilon\,
 \varphi(    c_\varepsilon t) e^{     \delta t}   \;\exp\left( -|\psi(y)-\psi(x)|+
(1+\eps)\Lambda(\psi)^2     t \right)
$$
 for all $t>0$, and  $x,y\in E\setminus \NN$. This proves \eqref{e:1.5}.  Furthermore, by
optimizing with respect to $\psi\in \sF_b$ with $\Lambda(\psi)^2<\infty$ at the right hand side of \eqref{e:1.5}, we can get \eqref{e:1.5--}.

When $(\sE, \sF)$ is a strongly local Dirichlet form on $L^2(E; m)$,
we know from Example \ref{exm4.2} that for any $\psi\in
\sF_b$ with
$\Lambda (\psi)^2<\infty$,  $\psi\in \mathcal{A}(1; \phi,1,1)$, where
 $\phi(1, \psi,r)=r\Lambda (\psi)^2.$
With this at hand,
we deduce from Theorem \ref{T:thm1} (or from \eqref{e:4.19}) that (i) implies (iii)' in Corollary \ref{C:local}. \qed

 \begin{remark}\label{R:3.3} \rm
For anomalous symmetric Markov processes (for example, symmetric diffusions on some fractals that satisfy a sub-Gaussian estimate), the energy
measure $\Gamma(\cdot,\cdot)$ is singular
 with respect to the underlying reference measure $m$; see \cite{Ku, Hi, KM}.
 In particular, in these contexts,
the energy measure $\Gamma(u,  u)$ is singular with respect to $m$ unless $u$ is a constant,
and so
the off-diagonal estimates in Theorem \ref{T:thm1} (iii)  degenerate into the on-diagonal estimates.
However, motivated by \cite{MS1, MS2}, we can check that \eqref{e:01-} holds true by making use of
the so-called cutoff Sobolev inequality.
The reader is referred to
\cite[Proposition 2.4]{MS1}
and \cite[Lemma 3.5]{MS2}
for details on anomalous diffusions with walk dimension $d_w>2$ and on anomalous
heavy tailed random walks with
 walk dimension $d_w\ge 2$ in the $d$-set setting, respectively.
Note that, for those two processes,  their on-diagonal
heat kernels
 are of  polynomials decay
(i.e.,\ of the form $c \, t^{-d/d_w}$).
With Theorem \ref{T:3.2} at hand,
we can
extend Theorem \ref{T:thm1} and Corollary \ref{C:local} to
these singular settings, and can handle
anomalous diffusions, anomalous symmetric jump processes as well as anomalous symmetric diffusions with jumps in
general measure spaces with more general scaling functions;
see \cite{GT, CKW1, CKW2}.
\end{remark}

\noindent{\bf Proof of Corollary \ref{Cor1.4}.}\,
For any pre-compact open set $D$, consider the following Dirichlet semigroup associated with $(P_t)_{t\ge0}$:
$$P_t^Df(x)=\E^x (f(X_t){\mathbbm 1}_{\{\tau_D>t\}}),\quad f\in L^2(D;m)\,,t>0,$$ where $\tau_D$ is the
 first exit time from the open set $D$
 of the Hunt process $X=\{X_t, t\geq 0; \bP_x, x\in E\setminus \NN\}$ associated with the Dirichlet form $(\sE,\sF)$.

 Suppose that  (ii) in Theorem \ref{T:thm1} holds with $\delta=0$. Then, by \cite[Theorem 3.1]{BBCK},
we know that the semigroup $(P_t^D)_{t\ge0}$ has
  a heat kernel $p^D(t,x,y)$ defined on $(0,\infty)\times (D\setminus {\cal{N}})\times (D\setminus \cal{N})$ for some properly exceptional set $\cal{N}$. In the following, we extend $p^D(t,x,y)$ to $(0,\infty)\times D\times D$ by setting $p^D(t,x,y)=0$ when $x\in {\cal N}$ or $y\in \cal{N}$.
According to (ii) in Theorem \ref{T:thm1} with $\delta=0$ again and  the fact that $m(D)<\infty$, we know that $$\int_Dp^D(t,x,x)\,m(dx)<\infty$$ for all $t>0$, and so $P_t^D$ is a Hilbert-Schmidt operator for all $t>0$. Thus, the semigroup $(P_t^D)_{t\ge0}$ is compact. By general theory of semigroups for compact operators, there exists an orthonormal basis of eigenfunctions
 $\{ \phi_n : n \geq 1\}
\subset L^2(D;m)$ associated with corresponding eigenvalues $\{\lambda_n(D): n\geq 1\}$ satisfying  $0<\lambda_1(D)<\lambda_2(D)\le \cdots $ and $\lim_{n\to\infty}\lambda_n(D)=\infty$. In particular, for any $m$-a.e.\! $x,y\in D$ and $t>0$,
$$p^D(t,x,y)=\sum_{k=1}^\infty \exp\left(-\lambda_k(D)t\right)\phi_k(x)\phi_k(y).$$
 With this expression for $p^D(t,x,y)$, that (ii) implies (v) follows from the proof of  \cite[Theorem 2.2]{Gri}, also thanks to the property (ii) for the regular function given in Definition \ref{D1.1}.

It is trivial that (v) implies (iv). Concerning the implication of (iv)$\Rightarrow$ (ii) we refer to \cite[Lemma 5.4]{GH}. Indeed, according to the proof of \cite[Lemma 5.4]{GH},   \eqref{iso} yields that for any $f\in \sF$ with compact
support,
$$c_1\|f\|_2^2\Theta\big(c_2/\|f\|_2^2\big)\le \sE(f,f),\quad \|f\|_1=1,$$ where
 $c_1,c_2$ are positive constants independent of $f$. This with \eqref{e:Nth} finishes
 the proof. \qed

\section{Application and Examples}

\subsection{Application: off-diagonal Dirichlet heat kernel upper bounds}\label{section5.1}

 In this subsection, we assume that
\begin{itemize}
\item[(i)] $V(x,r)$ is a strictly positive function on $E \times \R_+$ such that $r\mapsto V(x,r)$ is
increasing on $\R_+$ for any fixed $x\in E$, and that
there exist constants $0<d_1\le d_2$, $0<C_1\le C_2$ and $C_3 \ge 1$ so  that
\begin{align}
\label{e:V2}
 C_1  \Big(\frac Rr\Big)^{d_1}\le \frac{V(x,R)}{V(x,r)} \leq   C_2  \Big(\frac Rr\Big)^{d_2} \quad \hbox{for all } x\in E \hbox{ and }
0<r\le R
\end{align}
and
\begin{align}
\label{e:V1}
V(x, r) \le C_3 V(y, r)  \quad
  \text{for all } r>0 \hbox{ and }  x, y \in E \hbox{ with }  \rho(x,y)<r.
\end{align}
\item[(ii)]  Let $\phi: \bR_+\to \bR_+$ be a strictly increasing continuous
function such that
 $\phi (0)=0$ and there exist constants $C_*,\beta_1>0$ so that
\begin{align}
\label{e:pW}
 \frac{\phi (R)}{\phi (r)}    \ge  C_* \Big(\frac Rr\Big)^{\beta_1}
 \quad \hbox{for all }
0<r \le R.
\end{align}
\end{itemize}

\medskip

For any $x\in E$ and $R>0$, set \begin{equation}\label{e:beta}\beta_{x ,R}(r):=\frac{1 }{V(x,R)}\max\left\{\bigg(\frac{R}{\phi^{-1}(r)} \bigg)^{d_2},  \bigg(\frac{R}{\phi^{-1}(r)} \bigg)^{d_1}\right\}, \quad r>0.\end{equation}
Note that $\beta_{x,R}:\R_+\to \R_+$ is decreasing with $\beta_{x,R}(0)=
\infty$ and $\beta_{x,R}(\infty)=0$. By some elementary calculations, we can show that
for all $x \in E$ and $R>0$,
$1/\beta_{x,R}$ has the doubling property uniformly such that
for any $0<s\le t$,
\begin{equation}\label{e:5.5}
\frac{\beta_{x,R}(s)}
{\beta_{x,R}(t)}
   \le C_*^{-1/\beta_1} (t/s)^{d_2/\beta_1}.
 \end{equation}
This is because, when
$0< s\leq t \leq \phi( R)$,
$$
\frac{\beta_{x,R}(s)}{\beta_{x,R}(t)}
  =  \frac{\phi^{-1}(t)^{d_2}}{\phi^{-1}(s)^{d_2}}
 \le C_*^{-1/\beta_1} (t/s)^{d_2/\beta_1};
  $$
 when $0< s <\phi( R)<t$,
$$
\frac{\beta_{x,R}(s)} {\beta_{x,R}(t)}
= R^{d_2-d_1} \frac{\phi^{-1}(t)^{d_1}}{\phi^{-1}(s)^{d_2}}
  \le \frac{\phi^{-1}(t)^{d_2}}{\phi^{-1}(s)^{d_2}}
 \le C_*^{-1/\beta_1} (t/s)^{d_2/\beta_1};
 $$
when $\phi( R)\leq s \leq t$,
$$
\frac{\beta_{x,R}(s)}
{\beta_{x,R}(t)} =  \frac{\phi^{-1}(t)^{d_1}}{\phi^{-1}(s)^{d_1}}
\le C_*^{-1/\beta_1} (t/s)^{d_1/\beta_1} \leq C_*^{-1/\beta_1}
 (t/s)^{d_2/\beta_1}.
 $$

Define
$$
\Psi(r)=\int_r^\infty \frac{\beta_{x,R}^{-1}(s)}{s} \,ds,\quad r>0.
$$
   We claim that
 there is a constant
  $c_1:=c_1(C_*, \beta_1, d_2)\ge1$ (independent of $x$ and $R$) such that for all $r\in \R_+$,
\begin{equation}\label{e:5.6}
\Psi(r)\le c_1\beta_{x,R}^{-1}(r).
\end{equation}
 Consequently,   by \eqref{e:5.5} there is a constant $c_2:=c_2(C_*, \beta_1, d_2)>1$ (independent of $x$ and $R$) so that
\begin{equation}\label{e:5.61}
 \Psi^{-1} (s) \leq \beta_{x,R} (s/c_1)
 \leq c_2 \beta_{x,R} (s )
\quad \hbox{for } s>0.
\end{equation}
Indeed, \eqref{e:5.5} implies that $0<s\le t$,
  $$
   \frac{ \beta_{x,R}^{-1} (t) }{ \beta_{x,R}^{-1} (s) }  \le  C_*^{ -1/d_2} (s/t)^{\beta_1/d_2}.
 $$
Thus,
\begin{align*}
\Psi (r) &= \sum_{k=0}^\infty \int_{2^k r}^{2^{k+1} r} \frac{\beta_{x,R}^{-1}(s)}{ s } \,ds  \leq
 \sum_{k=0}^\infty \frac{\beta_{x,R}^{-1}(
 2^{k}r)}{ 2^k r}
( 2^{k+1} r  -2^kr ) \\
&\leq
 C_*^{-1/d_2}
 \beta_{x,R}^{-1}(r)  \sum_{k=0}^\infty 2^{-k\beta_1/d_2} =
\frac{C_*^{-1/d_2}} {1-2^{-\beta_1/d_2}}
\beta_{x,R}^{-1}(r).
\end{align*}
This establishes the claim \eqref{e:5.6}.

\medskip

Suppose that $(\sE,\sF)$ is a symmetric regular Dirichlet form on $L^2(E;m)$
having no killings inside $E$, and that $(P_t)$ is its associated semigroup. For an open set $B\subset E$, let $(P_t^B)_{t\ge0}$ be
the Dirichlet heat semigroup associated with $(P_t)_{t\ge0}$ with the Dirichlet conditions on $B^c$, i.e.,
$$P^B_tf(x)=\bE_x[f(X_t){{\mathbbm 1}}_{\{t< \tau_B\}}],\quad f\in L^2(B;m), \,\,t>0,\,\,x\in B,
$$
 where $\tau_B$ is the first exit time from the open set $B$ by
 the Hunt process $X=\{X_t, t\geq 0; \bP_x, x\in E\setminus \NN\}$ associated with the Dirichlet form $(\sE,\sF)$.

\begin{thm}\label{T:4.1}
Assume that
\eqref{e:V2}, \eqref{e:V1} and  \eqref{e:pW} hold.
Suppose that the semigroup $(P_t)_{t\ge0}$ has
   a heat kernel $p(t,x,y)$ defined on $(0,\infty)\times (E\setminus \NN)\times (E\setminus \NN)$ with
\begin{equation}\label{e:4.1} p(t,x,x)\le \frac{c}{V(x,\phi^{-1}(t))},\quad t>0, \,\, x\in E\setminus \NN,\end{equation} where $c>0$ is independent of $x$ and $t$. Then, for any $x_0\in M$ and $R>0$,
 the Dirichlet heat semigroup
 $(P_t^{B(x_0, R)})_{t\ge0}$
 has  a heat kernel
 $p^{B(x_0, R)} (t,x,y)$
  defined on $(0,\infty)\times (E\setminus \NN)\times (E\setminus \NN)$ such that
  \begin{itemize}
  \item[{\rm(i)}]  $p^{B(x_0, R)} (t,x,y)=0$ if either  $x$ or $y$ is in $ B (x_0, R)^c \setminus \NN ;$
\item[{\rm(ii)}] for any $\varepsilon>0$ and $\psi\in {\mathcal A}(C_0;\phi, C,\eta)$,
there is a constant $C_\varepsilon:=C_\varepsilon(C,\eta,d_1,d_2, C_0, C_1,C_2,$ $C_3, C_*, \beta_1, c)>0$
 such that for all $t>0$ and $x,y\in B (x_0, R)\setminus \NN,$
$$
p^{B(x_0, R)} (t,x,y)\leq C_\varepsilon\,\beta_{x_0,R}(t) \,\exp\left( -|\psi(y)-\psi(x)|+
(1+\varepsilon)   \phi(C_0, \psi,1)     t \right),
$$
where  $ \beta_{x_0,R}(r)$ is defined by \eqref{e:beta}.
In particular, for any $\varepsilon>0$, there is a constant
 $C_\varepsilon:=C_\varepsilon(d_1,d_2, C_1,C_2,$ $C_3, C_*, \beta_1, c)>0$
so that for all $t>0$, $x_0\in E$, $R>0$ and $x,y\in B(x_0,R)\setminus {\cal N}$,
$$
p^{B(x_0, R)} (t,x,y)\leq C_\varepsilon\,\beta_{x_0,R}(t) \,\exp\left( -|\psi(y)-\psi(x)|+
(1+\eps)\Lambda(\psi)^2     t \right),
$$ where $$\Lambda (\psi)^2 :=
  \max\left\{\left\|\frac{d e^{-2\psi}\Gamma(e^{\psi}, e^\psi)}{d m}\right\|_\infty,
  \  \left\|\frac{d e^{2\psi}\Gamma(e^{-\psi}, e^{-\psi})}{d m}\right\|_\infty\right\} <\infty.$$ \end{itemize}
\end{thm}

\medskip

To prove Theorem \ref{T:4.1}, we need to establish functional inequalities for the Dirichlet form $(\sE, \sF_B)$ associated with
 $(P_t^B)_{t\ge0}$, where
$ \sF_B=\{ u\in \sF:  u=0 \,\,\EE \hbox{-q.e. on } B^c \}$.
\begin{lem}
Suppose that  \eqref{e:V2}, \eqref{e:V1} and  \eqref{e:4.1} hold.
Then, there exists a constant
 $\wh C:=\wh C(c, C_1,  C_2, C_3)>0$ such that
 for any $x_0\in
E$ and $R>0$,
the following super-Poincar\'e inequality
 for $(\sE, \sF_{B(x_0,R)})$ holds{\rm\,:}
\begin{equation}\label{SPI}
     \left\| u \right\|^2_2 \le r \sE(u,u)+\wh C\beta_{x_0,R}(r)\left\| u \right\|^2_1,\quad r>0,\, u\in \sF_{B(x_0,R)},
\end{equation}
where $ \beta_{x,R}$ is defined by \eqref{e:beta}.
\end{lem}

\begin{proof}
By the Cauchy-Schwarz inequality,   for any $x,y\in E\backslash \NN$ and $t>0$,
$$
p(t,x,y)\le (p(t,x,x)p(t,y,y))^{1/2} \le \max\left\{ p(t,x,x),p(t,y,y)\right\}.
$$
This along with \cite[Theorem 2.1]{Ki} and \eqref{e:4.1} yields that for all $u\in \sF\cap L^1(E;m)$ and $t>0$,
\begin{align*}\|u\|_2^2&\le t
\sE (u, u)
+\|u\|_1^2\sup_{x\in{\rm supp}(u)} p(t,x,x)\\
&\le t \sE(u,u)+\|u\|_1^2 \, \frac{c }{\inf_{x \in {\rm supp}(u)}V(x,\phi^{-1}(t))}.  \end{align*}
Consequently, for all $u\in \sF\cap L^1(E;m)$ and $s>0$,
\begin{equation}\label{e:nash}
  \left\| u \right\|^2_2   \le  \phi(s)\sE(u,u)+  \frac{c \left\| u \right\|^2_1}{\inf_{z \in {\rm supp}(u)}V(z,s)}.
\end{equation}

 For fixed $x_0\in E$ and $R>0$, set $B=B(x_0,R)$.
 By
    \eqref{e:V2} and \eqref{e:V1}, we have that
   for any $z \in B$ and $s>0$,
\begin{align*}
  \frac{V(x_0,R)}{V(z,s)}\le& C_3
   \frac{V(z,R)}{V(z,s)}
  \le  C_3 (C_1^{-1} \vee C_2)\max\left\{\Big(\frac{R}{s} \Big)^{d_2},  \Big(\frac{R}{s} \Big)^{d_1}\right\}.\end{align*}
  Combining this with \eqref{e:nash}, we obtain that for any
  $u\in \sF\cap C_c(B)\subset \sF_B$,
  and $s>0$,
\begin{align*}
   \left\| u \right\|^2_2
   &\le  \phi(s)\sE(u,u)+ \frac{
   \wh C }{V(x_0,R)}\max\left\{\bigg(\frac{R}{s} \bigg)^{d_2},  \bigg(\frac{R}{s} \bigg)^{d_1}\right\}\left\| u \right\|^2_1 \\
   &=     \phi(s)\sE(u,u)+ \wh C \beta_{x_0,R}(s) \| u\|_1^2,
\end{align*}
where
$\wh C=c C_3 (C_1^{-1} \vee C_2).$
As $\sF_b \cap C_c(B)$ is $\sE_1$-dense in $\sF_B$ and uniformly dense in $C_c(B)$,
the desired assertion follows.
\qed
\end{proof}

\noindent {\bf Proof of Theorem \ref{T:4.1}  }\,\,
Note that the existence of  $p^{B(x_0, R)} (t,x,y)$  as well as the property (i) of  $p^{B(x_0, R)} (t,x,y)$   follow from the existence of  $p(t,x,y)$ and the strong Markov property of the process $X$.
So it remains to establish property (ii).

In the following, we fix $x_0\in E$ and $R>0$, and set $B=B(x_0,R)$. For any $\psi\in {\mathcal A}(C_0;\phi, C,\eta)$, non-negative
$f\in \sF_b$, $t>0$ and $p\ge1$,
$$
f_{t}(x): = P_t^{B,\psi}f(x)=e^{\psi(x)} P^B_t (e^{-\psi}f)(x),\quad
f_{p,t}(x)
 := \exp \left(-  \phi(C_0, \psi,p)
  t \right) f_t(x), \quad x \in B.
$$
Note that $f_t\in \sF_B$ and so $f_{p,t} \in \sF_B$. According to \eqref{e:01-} and
\eqref{SPI},
 for any $p\ge1$, $t>0$ and $r>0$,
\begin{align*}
\frac{d \|f_{p,t}\|_{2p}^{2p}}{dt}= & -
2p   \phi(C_0, \psi,p)\|f_{p,t}\|_{2p}^{2p}-2p \sE(e^{\psi} (f_{p,t})^{2p-1}, e^{-\psi} f_{p,t})\le
{-} 2
C_0 \sE(f_{p,t}^p,f_{p,t}^p)\\
\le& 2C_0\left(-\frac{1}{r}\|f_{p,t}\|_{2p}^{2p}+\frac{
\wh C\beta_{x_0,R}(r)}{r}\|f_{p,t}\|_{p}^{2p}\right).
\end{align*}
Taking
$$r=
\beta_{x_0,R}^{-1}\left(\frac{\|f_{p,t}\|_{2p}^{2p}}{2\wh C\|f_{p,t}\|_p^{2p}}\right),$$
we have
\begin{equation}
 \label{e:3.55}
\frac{d \|f_{p,t}\|_{2p}^{2p}}{dt}\le -\frac{C_0\|f_{p,t}\|_{2p}^{2p}}{
\beta_{x_0,R}^{-1}\left(\frac{\|f_{p,t}\|_{2p}^{2p}}{2\wh C\|f_{p,t}\|_p^{2p}}\right)}<0
\quad \hbox{for every } t>0.
\end{equation}
Define
$$
\Psi(r)=\int_r^\infty \frac{
\beta_{x_0,R}^{-1}(s)}{C_0s}\,ds, \quad r>0.
$$
 Since,
 in view of \eqref{e:5.5},  $t\mapsto  1/\beta_{x_0,R} (t)$ has the doubling property uniformly  for all $x_0 \in E$ and $R>0$,
by Lemma \ref{L2.5} and its proof,
\eqref{e:con-hk} hold for $\beta_{x_0,R}$
with common constants
$c(\lambda)\equiv 1$ and $C(\lambda)$
  for all $x_0 \in E$ and $R>0$.
 With this fact,  \eqref{e:5.6}, \eqref{e:5.61} and \eqref{e:3.55},
 one can repeat the proof of Proposition \ref{T:3.1} to see that, for any $\varepsilon>0$, there is  a constant
 $C_\varepsilon:=C_\varepsilon(C,\eta,d_1,d_2, C_0, C_1,C_2,$ $C_3, C_*, \beta_1, c)>0$
 (uniformly on $x_0$ and  $R$) so that $$\|P_t^{B,\psi}\|_{L^1(B;m)\to L^\infty(B;m)}\le
 C_\varepsilon\, \beta_{x_0,R}(t)\,  \exp\left(
   (1+\varepsilon)  \phi(C_0, \psi,1)
    t
   \right),\quad t>0.$$  Therefore, the desired property (ii) for $p^{B(x_0, R)} (t,x,y)$ follows from the proofs of Theorems \ref{T:3.2} and \ref{T:thm1}.
\qed

\subsection{Examples}\label{S:5.2}

In this part, we give  some examples to illustrate the applications of Theorem \ref{T:thm1} and Corollary \ref{C:local}.
\begin{example}\label{Exm1.2}
{\rm[}{\it Heat kernel for Brownian motion on manifolds}{\rm]}
\rm Let $(M,\rho) $ be a non-compact manifold with bounded geometry (i.e., $M$ has a positive injectivity radius $r_0$, and its Ricci curvature
is bounded from below) and with the Riemannian volume measure $m$.
Let $p(t,x,y)$ be the heat kernel on $M$; that is, it is the smallest positive fundamental solution to the heat equation
$$\frac{\partial u}{\partial t}=\Delta u$$ on $\R_+\times M$. It was proven in \cite[Theorem 1.1]{BCG} that, if ${\rm dim}M=d$ and there is a continuous positive strictly increasing function $v(r)$ on $[r_0,\infty)$ such that for all $x\in M$ and $r\ge r_0$,
$m(B(x,r))\ge v(r)$,
 then there are constants $C,c>0$ such that for all $t>0$ and $x,y\in M$,
$$
p(t,x,y)\le \varphi(t):=
\begin{cases}
Ct^{-d/2},\quad
&t<  r_0^2,  \\
C/\gamma(ct)  ,   \quad& t\ge r_0^2,
\end{cases}
$$
where $\gamma(t)$ is defined by
$$
t =r_0^2 + \int_{v(r_0^2)}^{\gamma(t)} v^{-1}(s)\,ds,\quad t\ge r_0^2.
$$
Now  suppose that $\gamma(r)$ is an increasing  doubling function on
$ [r_0,\infty)$  with $\gamma (\infty):=\lim_{r\to \infty} \gamma(r)=\infty$. By Proposition \ref{P:dou}, there exist  $\bar\varphi\in \mathcal{R}$ and a constant $c\ge1$ such that
$$c^{-1}\bar\varphi(r)\le \varphi(r)\le
c\bar\varphi(r),\quad r\in \R_+.$$
Then, according to Corollary \ref{C:local}, for any $\eps>0$, there is a constant $C_\varepsilon>0$ so that for all $x,y\in M$ and $t>0$,
$$p(t,x,y)\le C_\varepsilon\widetilde\varphi(t)
 \exp \left(-\frac{\rho(x,y)^2}{4(1+\eps)t} \right),
$$
 where
$$ \widetilde\varphi(t):=
\begin{cases}
 t^{-d/2},\quad &t\le r_0^2,   \\
{\gamma(t)^{-1}},\quad& t> r_0^2.
  \end{cases}$$

Note that, when $v(r)=c_1e^{c_2r^\alpha}$ with $\alpha\in (0,1]$ on $[r_0,\infty)$,
it holds that for large $t$,
$c_3^{-1}\frac{t}{(\log t)^{1/\alpha}}\le \gamma(t)\le c_3\frac{t}{(\log t)^{1/\alpha}}$
with some $c_3\ge1$;
when $v(r)=c_4r^\beta$
with $\beta\ge1$ on $[r_0,\infty)$ (due to the fact that for any non-compact manifold $M$ with bounded geometry the volume growth is at least linear),
it holds that for large $t$,
 $c_5t^{\beta/(1+\beta)}\le \gamma(t)\le c_5t^{\beta/(1+\beta)}$ for some $c_5\ge1$.
In both cases $\gamma(r)$ are doubling functions such that $\gamma(\infty)=\infty$.

We also note that one may apply \cite[Theorem 1.1]{Gri} to obtain off-diagonal Gaussian estimates for this example as well. \qed\end{example}

\begin{example}{\rm[}{\it Heat kernel for Brownian motion  on hyperbolic spaces and its subordination}{\rm]} \rm Suppose that the Nash-type inequality \eqref{E:Nash} holds with $\delta=0$ and
$$
\theta(r)=\max\left\{r^{(\alpha+1)/\alpha},  \   r\log^{(\gamma-1)/\gamma}(2+r^{-1}) \right\}
$$
for some $\alpha>0$ and $\gamma\in (0,1]$. Then, by Theorem \ref{T:thm1} and the table in the end of Section \ref{Section3}, the on-diagonal upper bounds \eqref{e:ondup0} and the upper bounds \eqref{e:1.5--} hold with
$\varphi(t)=t^{-\alpha} e^{-t^\gamma}$ for the heat kernel of the associated semigroup $(P_t)_{t\ge0}$.

 A typical example is
 Brownian motion $(B_t)_{t\ge0}$ in the $d$-dimensional hyperbolic space ${\mathbb H}_k^d$,
  which is a simply connected complete $d$-dimensional Riemannian manifold
with a constant negative sectional curvature $-k^2$. In particular,
for $d=3$,
the assertion above holds with $\alpha=3/2$ and $\gamma=1$; e.g., see \cite[(1.3)]{Gri}.
The reader is also referred to \cite[Theorem 1.1]{GN}  for the explicit formulas of
the heat kernel on
 ${\mathbb H}_1^d$ for all $d\ge1$.

Let $(S_t)_{t\ge0}$ be a subordinator
that is independent of $(B_t)_{t\ge0}$ and has the Laplace exponent $h$.
Consider the subordinated process $(X_t)_{t\ge0}:=(B_{S_t})_{t\ge0}$ on the 3-dimensional
 hyperbolic space ${\mathbb H}_k^3$.
 By  \cite[Theorem 1]{SW},  the Nash-type inequality \eqref{E:Nash} holds with $\delta=0$ and
$\theta(r)=\max\{r h(r^{2/3}), r\}$ for  the Dirichlet form associated with
the subordinated process $(X_t)_{t\ge0}$.
In particular, let $(S_t)_{t\ge0}$ be a $\beta$-stable subordinator whose
Laplace exponent
 is $h(r)=r^\beta$ for some $\beta\in(0,1)$.
Then, \eqref{E:Nash} holds with $\delta=0$ and
$\beta(r)=\max\{r^{(1+(2\beta/3)}, r\},$  and so, by Theorem \ref{T:thm1}, the on-diagonal upper bounds \eqref{e:ondup0} and the upper bounds \eqref{e:1.5--} hold with
$\varphi(t)=t^{-3/(2
\beta)} e^{-t}$ for this $\beta$-stable subordinated Brownian motion on the 3-dimensional hyperbolic space. This gives us a concrete example of symmetric jump process whose heat kernel
 decays  exponentially
for large time.  \end{example}

It is now known that Davies' method presented in \cite{CKS} is very useful to obtain off-diagonal upper bounds for
heat kernels of symmetric jump processes, e.g., see \cite{BGK, CKK, CK03, CK08, CK10}.
All the works  mentioned above  except \cite{CK08, CK10}
  around Davies' method heavily relied on the assumption that the heat kernel $p(t,x,y)$ enjoys the
polynomial decay
in on-diagonal estimate (that is, \eqref{e:1-01} holds with $\varphi(t)=c \, t^{-\nu}$ for some $c,\nu>0$).

In \cite{CK08} (see also \cite{CK10, CKW1}),
the correct
 on-diagonal bound is  of the form $ c_0(\Phi^{-1}(t)) ^{-d}$ for some strictly increasing weighted function
 $\Phi:
 \R_+\to \R_+$ which satisfies doubling and reverse doubling properties; in particular,  it is not necessarily
of power function type and so   one cannot directly apply \cite[Theorem (3.25)]{CKS} to get heat kernel off-diagonal upper bound.
The approach of \cite{CK08}   is first to replace $ c_0(\Phi^{-1}(t)) ^{-d}$ by a rough upper bound $c_1 t^{-\theta}$
with some $c_1$, $\theta >0$ for $t\in (0, 1]$, and use
 \cite[Theorem (3.25)]{CKS} to obtain a preliminary off-diagonal upper estimate, and then bootstrap it to a precise off-diagonal estimate by a suitable scaling procedure.
However, this approach does not work in some other settings such as reflected diffusions with jumps in \cite{CKKW}.

 Below we give an example of  symmetric jump process on $\bR^d$ to show that some
 estimates  in \cite{CK08}
can be directly derived  by using Theorem \ref{T:thm1}. Note that, the approach of the example
below can be easily extended to metric measure spaces setting with uniformly volume doubling and reversed doubling properties.
 In particular, with Theorem \ref{T:thm1} in hand, we can obtain \cite[Lemma 4.3]{CK08}
directly without using a  scaling argument.

\begin{example}\label{Exm1.4}{\rm [}{\it Symmetric L\'evy-like processes with general scaling functions}{\rm]}
\rm
Consider a regular symmetric non-local Dirichlet form $(\sE,\sF)$ on $L^2(\R^d;dx)$ as follows:
\begin{align*}
\sE(u,v)&=\frac{1}{2}\iint_{\bR^d\times \bR^d\setminus {\rm diag}}(u(x)-u(y))(v(x)-v(y)) J(x,y) \,d x \,d y,\\
\sF &=\{u\in L^2(\bR^d;dx):\sE(u, u)<\infty\}, \end{align*} where  \begin{align}\label{e:intcon}
 \frac{c^{-1}}{|x-y|^d\phi(|x-y|)}\leq  J(x,y)\leq \frac{c}{|x-y|^d\phi(|x-y|)},\quad x\neq y
\end{align}
with some $c\ge1$ and a strictly increasing function
$\phi:\R_+\to\R_+$
satisfying that
$$
\int_0^1 \frac{s }{\phi(s)} \,ds<\infty
$$
and
\begin{align}\label{e:psiW}
 C_L(R/r)^{\beta_1} \le \phi(R)/\phi(r) \le C_U(R/r)^{\beta_2}, \quad 0<r<R<\infty
\end{align}
for some $C_L, C_U>0$, $\beta_1 \in (0, 2)$ and  $\beta_2 \ge \beta_1$.
It was proven in \cite[Theorem 3.4]{BKKL} that the Nash-type inequality \eqref{e:1-02} holds with $\theta(r)=c_0r/\Phi(r^{-1/d})$, where
$$\Phi(r)=\frac{r^2}{2\int_0^r \frac{s}{\phi(s)}\,ds}.$$
 We note that, by \cite[Section 2.1]{BKKL},  $\Phi:\R_+\to\R_+$ is an increasing function such that
\begin{align}\label{d:Pp}
\Phi(r) < \phi(r), \quad r\in \R_+
\end{align}
and
\begin{align}\label{e:PW}
 C_L(R/r)^{\beta_1} \le \Phi(R)/\Phi(r) \le  (R/r)^{2}  \wedge  (C_U(R/r)^{\beta_2}), \quad 0<r<R<\infty.
\end{align} In particular, $\Phi$ has the doubling property on $\R_+$ with $\Phi(0)=0$ and $\Phi(\infty)=\infty$.
Therefore, according to Proposition \ref{T:2.1}(i) and \cite[Theorem 3.1]{BBCK}, there are a properly exceptional set $\NN$ and
  a heat kernel
$p (t,x,y)$ associated with the Dirichlet form $(\sE,\sF)$ and defined on $(0,\infty)\times (\R^d\backslash \NN)\times (\R^d\backslash \NN)$ such that for any $t>0$ and $x,y\in \R^d\backslash \NN$,
\begin{equation}\label{e:uhkd} p(t,x,y)\le \frac{c_1}{(\Phi^{-1}(t))^d}
\end{equation}
for some $c_1>0$.

To consider off-diagonal upper bounds for $p(t,x,y)$, we will adopt the truncated argument as in \cite{BGK, CKK, CK03, CK08}.
For each $\rho>0$, we  define  a bilinear form $(\sE^{(\rho)}, \sF)$ by
\begin{align*}
\sE^{(\rho)}(u,v)=\frac{1}{2}\iint_{\bR^d\times \bR^d} (u(x)-u(y))(v(x)-v(y)){{\mathbbm 1}}_{\{|x-y|\le \rho\}}\, J(x, y)\,dx\,dy.
\end{align*}
It follows from \eqref{e:intcon} and \eqref{e:psiW} that there exists a constant $c_2>0$ such that for all $\rho>0$ and $u\in \sF$,
\begin{align*}
 \EE(u,u)
 &=
 \sE^{(\rho)}(u,u) +
 \frac{1}{2}
\int_{\{|x-y| \ge \rho\}} (u(x)-u(y))^2J(x, y)\, dx\,dy\\
  &\le  \sE^{(\rho)}(u,u) +
   2\left(\sup_{x\in \R^d} \int_{\{|x-y|\ge \rho\}}J(x,y)\,dy \right)\|u\|_2^2\\
  &\le \sE^{(\rho)}(u,u)+\frac{c_2}{\phi(\rho)}\|u\|^2_2.\end{align*} This implies that
$(\sE^{(\rho)},
\sF)$ is also a regular symmetric Dirichlet form on $L^2(\bR^d; dx)$, and that the Nash-type inequality \eqref{E:Nash} holds for $(\sE^{(\rho)},
\sF)$ with $\theta(r)=c_0r/\Phi(r^{-1/d})$ and $\delta={c_2}/{\phi(\rho)}$. As mentioned above, $\Phi$ has the doubling property on $\R_+$ such that $\Phi(0)=0$ and $\Phi(\infty)=\infty$, and so $r\mapsto \frac{1}{[\Phi^{-1}(r)]^d}$ is a regular function, thanks to Proposition \ref{P:dou}.
According to Theorem \ref{T:thm1}, there is  a heat kernel
 $p^{\rho}(t,x,y)$ associated with the Dirichlet form $(\sE^{(\rho)},\sF)$ and defined on $(0,\infty)\times (\R^d\backslash \NN)\times (\R^d\backslash \NN)$ such that for any $t>0$, $x,y\in \R^d\backslash \NN$  and $\psi\in \sF_b$, \begin{align*}
 p^\rho(t,x,y)  \le \frac{c_3e^{{c_2t}/{\phi(\rho)}}}{ (\Phi^{-1}(t))^d}  \exp\Big( -|\psi(x)-\psi(y)|+c_4t\big(\| \Gamma_\rho[{\psi}] \|_\infty \lor \| \Gamma_\rho[{-\psi}] \|_\infty\big)  \Big),
\end{align*} where $$\Gamma_\rho[\psi](x):=
\frac{1}{2}\int_{\{|x-y| \le \rho\}} (e^{\psi(x)-\psi(y)}-1)^2 J(x,y)\,dy. $$
Here, for simplicity we take the same properly exceptional set $\NN$ as before.

Next, we fix $x,y \in \R^d\setminus \sN$ and  take
$$
\psi(z):= \frac{s}{3} (|z-x| \land |x-y|) \quad \mbox{for } z \in \R^d,
$$
where $s>0$ is a constant to be chosen later. By the facts that $(1-e^r)^2 \le r^2 e^{2|r|}$ for $r \in \bR^1$ and $$|\psi(z_1)-\psi(z_2)| \le \frac{s}{3} | |z_1-x|-|z_2-x| | \le \frac{s}{3} |z_1 -z_2|,\quad z_1, z_2 \in \R^d,$$ it holds for every $z\in \R^d$ that
\begin{align*}
\Gamma_\rho[{\psi}](z) &= \frac{1}{2}\int_{\{|z-w| \le \rho\}} (1- e^{\psi(z)-\psi(w)})^2 J(z,w)\, dw \\	
&\le \frac{1}{2}\int_{\{|z-w| \le \rho\}} (\psi(z)-\psi(w))^2 e^{2|\psi(z)-\psi(w)|} J(z,w)\,dw \\
&\le \frac{1}{2}\left(\frac{s}{3}\right)^2 e^{2s\rho/3} \int_{\{|z-w| \le \rho\}} |z-w|^2 J(z,w)\,dw \\
&\le c_5 s^2 e^{2s\rho/3}
\int_0^{\rho} \frac{v}{\phi(v)}\,dv
=\frac{c_5}{2} s^2 e^{2s\rho/3} \frac{\rho^2}{\Phi(\rho)} \le c_{6} \frac{e^{s\rho}}{\Phi(\rho)},
\end{align*}
where in the fourth inequality we used \eqref{e:intcon}, and the last inequality follows from the inequality $v^2e^{2v/3}\leq 9e^v$ for $v>0$.
Thus, $$ \big\| \Gamma_\rho[{\psi}] \big\|_\infty \lor \big\| \Gamma_\rho[{-\psi}] \big\|_\infty \le c_{6} \frac{e^{s\rho}}{\Phi(\rho)}.$$
Therefore, we arrive at that for any $t>0$, $x,y\in \R^d\backslash \NN$,
\begin{equation}
\label{e:q4}
 p^\rho(t,x,y) \le \frac{c_3e^{{c_2t}/{\psi(\rho)}}}{(\Phi^{-1}(t))^d}   \exp \left( -\frac{s|x-y|}{3} + c_6 \frac{te^{s\rho}}{\Phi(\rho)} \right).
\end{equation}

We now further assume that  $t \le \Phi(|x-y|)$.
Define $\gamma:=\frac{\beta_1}{3(d+\beta_1)}$. Take $\rho= \gamma |x-y|$ and $s = \frac{1}{\gamma |x-y|} \log (\frac{\Phi(|x-y|)}{t})$. Then, by \eqref{e:PW},
\begin{align*}
-\frac{s|x-y|}{3} +  c_6  \frac{t e^{s\rho} }{\Phi(\rho)} &= \frac{1}{3\gamma} \log\left(\frac{t}{\Phi(|x-y|)}\right) + c_6  \frac{\Phi(|x-y|)}{\Phi(\rho)}\le \frac{1}{3\gamma} \log\left(\frac{t}{\Phi(|x-y|)}\right) + \frac{ c_6 }{\gamma^2}.
\end{align*}
On the other hand, due to \eqref{d:Pp} and \eqref{e:PW},
\begin{align*}
  \frac{c_2t}{\phi(\rho)}\le
      \frac{c_2t}{\Phi(\rho)}\le \frac{c_2\Phi(|x-y|)}{\Phi(\rho)}\le   c_7.
\end{align*}
Thus,
combining both estimates above with \eqref{e:PW} and \eqref{e:q4}, we obtain that for any $t>0$ and $x,y \in \R^d\setminus \sN$,
\begin{align*}
p^\rho(t,x,y) &\le \frac{ c_{8}}{(\Phi^{-1}(t))^d}  \left( \frac{t}{\Phi(|x-y|)} \right)^{1/3\gamma}\nn \\
&=  c_{8}\left( \frac{\Phi^{-1}(\Phi(|x-y|))}{\Phi^{-1}(t)} \right)^d \left( \frac{t}{\Phi(|x-y|)} \right)^{1+d/\beta_1}    \frac{1}{|x-y|^d} \nn \\
&\le  c_{9}\left( \frac{\Phi(|x-y|)}{t} \right)^{d/\beta_1} \left( \frac{t}{\Phi(|x-y|)} \right)^{1+d/\beta_1}  \frac{1}{|x-y|^d} \nn \\
& =   \frac{c_{9} t}{ |x-y|^d\Phi(|x-y|)}.
\end{align*}
This along with \cite[Lemma 3.1]{BGK}, \eqref{e:intcon}, \eqref{d:Pp} and  \eqref{e:PW} yields that for every $0<t\le \Phi(|x-y|)$ and $x,y \in \R^d\setminus \sN$,
 \begin{equation}\label{e:q8}\begin{split}
p(t,x,y) \le \frac{ c_{9}  t}{|x-y|^d \Phi(|x-y|)} +
\frac{ct}{\phi(\gamma |x-y|) (\gamma |x-y|)^d}
 \le  \frac{ c_{10} t}{|x-y|^d \Phi(|x-y|)}.
\end{split}\end{equation}
We conclude from \eqref{e:uhkd} and \eqref{e:q8} that
for every $t>0$ and $x,y \in \R^d\setminus \sN$,
\begin{equation}\label{e:kk}
p(t,x,y) \le
c_{11}
 \left(\frac{1}{(\Phi^{-1}(t))^d}  \wedge  \frac{ t}{|x-y|^d \Phi(|x-y|)} \right).
\end{equation} \qed
 \end{example}

 \begin{remark}\rm
 \begin{itemize}
 \item[(i)] The proof above simplifies the proof of \cite[Theorem 3.8]{BKKL} without using the scaling argument.
Moreover,
either by following the same arguments
in the above proof or by \eqref{e:kk} with the same arguments as these in the proofs of \cite[Theorem 2.4]{CKK1} and \cite[Proposition 2.2]{CKS1},
one can show that, if we assume
$
\sup_{x \in \R^d} \int_{\{|y-x|\ge1\}} J(x,y)\,dy < \infty
$
and \eqref{e:intcon} holds only for $0<|x-y| \le 1$, then the upper bound in \eqref{e:kk} holds true for $0<t \le1$ and $x,y \in \R^d\setminus \sN$ with $|x-y| \le 1$.
Note that, if $\psi(r)=r^2\ln ^{1+\beta}(2/r)$ for all $r\in (0,1]$ with $\beta>0$, then  $c_1r^2\ln ^{\beta}(2/r)\le \Phi(r) \le c_2 r^2\ln ^{\beta}(2/r)$
for $0<r \le 1$.
 Thus, we
recover  the main result of \cite{Mi} (see \cite[Theorem 1.1]{Mi}) and in fact we get much more.

\item[(ii)] We should mention that the heat kernel upper bounds \eqref{e:kk} are sharp when $c^{-1}\phi(r)\le\Phi(r)\le c\phi(r)$ holds for all $r>0$ with some $c\ge1$, see \cite{CK08} for more details.
On the other hand, even when the estimate \eqref{e:kk} is not optimal in general setting,
 it is the first crucial step to
obtain
the optimal estimates, see \cite{BKKL, CKW3}.
\end{itemize} \end{remark}

\noindent {\bf Acknowledgement.} The authors would thank Alexander Grigor'yan for fruitful discussions and for giving related references on manifolds.
They also thank Naotaka Kajino for useful comments on
the proof of \cite[Theorem 3.2.4]{Dav2}.
The research of Zhen-Qing Chen is partially supported by Simons Foundation Grant 520542.
The research of Panki Kim is supported by
 the National Research Foundation of Korea (NRF) grant funded by the Korea government (MSIP)
(No.\ 2016R1E1A1A01941893).\ The research of Takashi Kumagai is supported
by JSPS KAKENHI Grant Number JP17H01093 and by the Alexander von Humboldt Foundation.\
 The research of Jian Wang is supported by the National
Natural Science Foundation of China (Nos.\ 11831014 and 12071076), the Program for Probability and Statistics: Theory and Application (No.\ IRTL1704) and the Program for Innovative Research Team in Science and Technology in Fujian Province University (IRTSTFJ).

\small

\bibliographystyle{plain}

\vskip 0.3truein

{\bf Zhen-Qing Chen}

Department of Mathematics, University of Washington, Seattle,
WA 98195, USA

E-mail: zqchen@uw.edu

\bigskip

{\bf Panki Kim}

Department of Mathematical Sciences
and Research Institute of Mathematics,
Seoul
National University,
\newline
\indent
Seoul 08826,
  Republic of Korea

E-mail: pkim@snu.ac.kr

\bigskip

{\bf Takashi Kumagai}

Research Institute for Mathematical Sciences,
Kyoto University, Kyoto 606-8502, Japan

E-mail: kumagai@kurims.kyoto-u.ac.jp

\bigskip

{\bf  Jian Wang}

College of Mathematics and Informatics \& Fujian Key Laboratory of Mathematical
Analysis and Appl-\newline \indent ications (FJKLMAA)  \&
 Center for Applied Mathematics of Fujian Province (FJNU),
Fujian Normal \newline \indent University,  Fuzhou 350007,  China

 Email: jianwang@fjnu.edu.cn

\end{document}